\documentclass[final]{siamltex}
\usepackage{amsmath, amssymb, graphicx, multirow}

\newcommand{\Eq}[1]{(\ref{eq:#1})}
\newcommand{\bTh}[1]{Th.~\ref{thm:#1}}
\newcommand{\Lem}[1]{Lem.~\ref{lem:#1}}

\newcommand{\Def}[1]{Def.~\ref{def:#1}}
\newcommand{\Sec}[1]{\S \ref{sec:#1}}
\newcommand{\Fig}[1]{Fig.~\ref{fig:#1}}

\newcommand{\App}[1]{Appendix~\ref{app:#1}}

\newcommand{\InsertFig}[4]
{\begin{figure}[h!t]
       \centerline{
         \includegraphics[width=#4]{./#1}
       }
       \caption{{\footnotesize  #2}
       \label{fig:#3}}
\end{figure}}

\newcommand{\InsertFigTwo}[5] {
\begin{figure}[h!t]
       \centerline{
         \includegraphics[width=#5]{./#1}
         \hskip 0.5in
         \includegraphics[width=#5]{./#2}
       }
       \caption{{\footnotesize  #3}
       \label{fig:#4}}
\end{figure}}

\newcommand{\bR}{{\mathbb{ R}}}

\newcommand{\bT}{{\mathbb{ T}}}
\newcommand{\bZ}{{\mathbb{ Z}}}

\newcommand{\cA}{{\cal A}}
\newcommand{\cB}{{\cal B}}

\newcommand{\cD}{{\cal D}}

\newcommand{\cF}{{\cal F}}
\newcommand{\cL}{{\cal L}}

\newcommand{\cR}{{\cal R}}
\newcommand{\cS}{{\cal S}}
\newcommand{\cU}{{\cal U}}

\newcommand{\cI}{{\cal I}}
\newcommand{\cP}{{\cal P}}
\newcommand{\cW}{{\cal W}}

\newcommand{\bc}{{\bf c}}

\newcommand{\bx}{{\bf x}}
\newcommand{\bw}{{\bf w}}

\newcommand{\vphi}{\varphi}

\newcommand{\Wu}{W^{\mathrm{u}}}
\newcommand{\Ws}{W^{\mathrm{s}}}
\newcommand{\Vol}{\mathrm{Vol}}

\newcommand{\sgn}{\mathop{\rm sgn}\nolimits}

\newcommand{\beq}[1]{\begin{equation}\label{eq:#1}}
\newcommand{\eeq}{\end{equation}}

\newenvironment{se}[1]{\equation\label{eq:#1}\aligned}{\endaligned\endequation}
\newcommand{\bsplit}[1]{\begin{se}{#1}}
\newcommand{\esplit}{\end{se}}

\newenvironment{example}[1][]
  {
	\setlength \leftmargini {1.0em}		
	\setlength \topsep {0.5em}			
	\begin{quote}
	{\it Example#1} }
	{\end{quote}
  }
\newcommand{\bexam}[1][:]{\begin{example}[#1]}
\newcommand{\eexam}{\end{example}}

\title{Transport in Transitory, Three-Dimensional, Liouville Flows \thanks
      {
        BAM and JDM were supported in part by NSF grant DMS-0707659.
        Useful conversations with Michel F.~M.~Speetjens are gratefully acknowledged.
      }}
\author{B.~A.~Mosovsky $^\dag$ and J.~D.~Meiss 
        \thanks{Department of Applied Mathematics, 
        University of Colorado, Boulder, CO 80309-0526 
        ({\tt brock.mosovsky@colorado.edu, james.meiss@colorado.edu})}
}
\date{\today}
\begin{document}
\maketitle

\begin{abstract}
We derive an action-flux formula to compute the volumes of lobes quantifying transport between past- and future-invariant Lagrangian coherent structures of $n$-dimensional, transitory, globally Liouville flows. A transitory system is one that is nonautonomous only on a compact time interval.  This method requires relatively little Lagrangian information about the codimension-one surfaces bounding the lobes, relying only on the generalized actions of loops on the lobe boundaries.  These are easily computed since the vector fields are autonomous before and after the time-dependent transition.  Two examples in three-dimensions are studied: a transitory ABC flow and a model of a microdroplet moving through a microfluidic channel mixer.  In both cases the action-flux computations of transport are compared to those obtained using Monte Carlo methods.
\end{abstract}



\section{Transitory Systems} \label{sec:intro}
Finite-time transitions between steady states are common in a wide range of physical systems.  They play a key role in industrial mixing processes, mechanical systems in which parameters are modulated in a time-dependent manner, chemical reactions that progress to equilibrium, and shifts in local ecology due to a sudden environmental change.  Since the transition mechanism may be complex and the starting and ending states often differ, a prediction of the final state of the system requires a detailed understanding of the transitional dynamics.

In many cases, an analysis of transport and mixing in these systems can provide such an understanding.  However, since any finite-time transition is aperiodically time-dependent, traditional techniques for computing dynamical transport \cite{MMP84, RomKedar88, Meiss92, MacKay94, Lomeli08c} are often insufficient.  In nonautonomous systems transport is often thought of as occurring between ``Lagrangian coherent structures"; these are variously defined, for example, using ridges of finite-time Lyapunov exponent fields \cite{Haller00, Shadden05, Lekien07}, distinguished hyperbolic trajectories \cite{Ide02, Madrid09}, or eigenfunctions of the Perron-Frobenius operator \cite{Froyland09, Froyland10}.  However, few studies have quantitatively computed transport between coherent structures in aperiodic flows \cite{Haller98b, Cardwell08, Mendoza10, Mosovsky11}, and these have been restricted to two-dimensions.  Several studies of mixing in aperiodic flows have also been conducted \cite{Liu94, Poje99, Kang08}; however, these have focused primarily on global mixing measures rather than transport between coherent structures, and again results have been restricted to 2D.  To the best of our knowledge, no studies to date have given a quantitative description of finite-time transport between isolated coherent structures in a 3D, aperiodic flow.

In this paper, we present a formalism to compute transported volumes between Lagrangian coherent structures in a class of 3D aperiodic flows that we call ``transitory''---the formal definition will be recalled from \cite{Mosovsky11} in the remainder of this section. Our theory applies to incompressible vector fields that are, in addition, globally ``Liouville", see \Sec{Liouville}. The Lagrangian coherent structures we consider are past- and future-invariant regions of phase space, and transport between them corresponds to the volumes of certain lobes comprising the intersections of these regions.  The salient point is that the computation of lobe volumes can be done by knowing only key ``heteroclinic" trajectories that lie on the lobe boundaries.  Compared to a na\"{\i}ve, volume-integral approach, our method reduces by two the dimension of the Lagrangian information needed at any instant in time to compute a lobe volume. The result is an action-flux formula for $n$-dimensional lobe volumes, see \Sec{lobeVolumes}.  As examples, we will compute transport in a nonautonomous version of Arnold's ABC flow \cite{Arnold65} in \Sec{abc}, and in a model flow of a droplet in a microfluidic mixer in \Sec{microdroplet}.

On a phase space $M$, a \emph{transitory} ODE  \cite{Mosovsky11} of \emph{transition time} $\tau$ has the form
\beq{transitoryODE}
    \dot \bx = V(\bx,t)  , \quad  
        V(\bx,t) = \left\{\begin{matrix} P(\bx), & t < 0 \\
			                 F(\bx), & t > \tau \end{matrix}\right. ,
\eeq
where $P: M \to TM$ is the past vector field, $F: M \to TM$ is the future vector field, and $V:M \times \bR \to TM$ is otherwise arbitrary on the transition interval $[0, \tau]$.  One way to model this type of behavior is by way of a transition function
\beq{transitionFunction}
  s(t) = \left\{ \begin{matrix} 0, &t < 0 \\ 
				1, &t > \tau
		  \end{matrix} \right. ,
\eeq
with the convex combination
\beq{transitoryModel}
  V(\bx,t) = (1-s(t)) P(\bx) + s(t) F(\bx).
\eeq
For example, choosing
\beq{cubicTransition}
  s(t) =  \frac{t^2}{\tau^2}\left(3-2\frac{t}{\tau} \right) \quad \textrm{for} \quad t \in [0,\tau],
\eeq
along with \Eq{transitionFunction}, implies that $V$ is $C^1$ in time; we will use this form of $s(t)$ in \Sec{abc}.

Since the nonautonomous portion of the dynamics of \Eq{transitoryODE} is assumed to occur on a compact interval, it can be effected by a map.  Suppose that $V$ in \Eq{transitoryModel} has a complete flow $\vphi_{t_1,t_0}: M \to M$ that maps a point from its position at $t=t_0$ to its position at $t=t_1$ for any $t_0, t_1 \in \bR$.  Given a set $\cA_{t_0} \subseteq M$ at time $t_0$, denote its evolution at time $t$ under the flow by 
\[ 
  \cA_t = \vphi_{t,t_0} (\cA_{t_0}),
\]
and its orbit in the extended phase space by
\[ 
  \cA = \{ (\cA_t,t) : t \in \bR \} \subseteq M \times \bR.
\]
The orbit, $\cA$ of any $\cA_{t_0} \subseteq M$ is clearly invariant under $\vphi$, and we refer to $\cA_t$ as the \emph{time-$t$ slice} of $\cA$.  The {\em transition map} $T: M \to M$ for \Eq{transitoryODE} is
\beq{transitionMap}
	T(\bx) = \vphi_{\tau,0}(\bx).
\eeq
Consequently, a set $\cA_0$ at $t=0$ becomes $\cA_\tau = T(\cA_0)$ at time $\tau$, and thereafter evolves under $F$.  If the dynamics of $P$ and $F$ are known, then the only nontrivial work we must do is to characterize the map $T$.

For transitory systems, it is natural to introduce some terminology for orbits according to their behavior under the stationary vector fields $P$ or $F$.  We will say an orbit $\cA$ is \emph{past invariant} if $\cA_t = \cA_0$  for all $t<0$, and such a set is \emph{past hyperbolic} if $\cA_0$ is a hyperbolic invariant set of $P$.  Similarly, $\cA$ is \emph{future invariant} if $\cA_t = \cA_\tau$ for all $t>\tau$ and is \emph{future hyperbolic} if $\cA_\tau$ is a hyperbolic invariant set of $F$.  By extension, we call a slice $\cA_t$ past/future invariant/hyperbolic if its orbit in the extended phase space satisfies the above definitions.  Of course, sets that are invariant/hyperbolic under $P$ or $F$ need not be so under the transitory vector field $V$.  For example, the orbit of a hyperbolic equilibrium $p$ of $P$ is both past invariant and past hyperbolic.  However, because of the time-dependence of $V$ on $[0,\tau]$, $T(p)$ is typically not an equilibrium of $F$, and even if it is, it need not be hyperbolic. Thus, the orbit of $p$ under the transitory flow $\vphi$ need not be future invariant nor future hyperbolic.  These concepts of ``half-time'' invariance and hyperbolicity will be used extensively in the remainder of the paper.

Recall that stable and unstable sets $W^{s,u}(\gamma) \subset M \times \bR$ of an orbit $\gamma \subset M \times \bR$, are the sets of points that approach $\gamma_t$ as $t \to +\infty$ or $-\infty$, respectively. When $\gamma$ is past hyperbolic, each time-$t$ slice of its unstable set for $t<0$ is the unstable manifold of the orbit of $\gamma_0$ under the stationary flow of $P$; more importantly, the flow of this manifold under $V$ is precisely the unstable manifold of the \emph{full} orbit $\gamma$. However, the stable manifold of the orbit of $\gamma_0$ under the flow of $P$ is not dynamically relevant for the transitory vector field---it almost certainly is not a stable set for $\gamma$. Thus, the unstable manifold of a past-hyperbolic set is dynamically relevant for the transitory vector field.  Similarly it is the stable manifold of a future-hyperbolic set that is dynamically relevant for $V$. In the application described in \Sec{abc}, the intersections between an unstable manifold of a past-hyperbolic set and a stable manifold of a future-hyperbolic set will be used to define lobes pivotal to the study of transport.

\section{Liouville vector fields}\label{sec:Liouville}

Recall that Hamiltonian systems are defined on even dimensional manifolds $M$ that are endowed with a closed, nondegenerate two-form $\omega \in \Lambda^2(M)$, the ``symplectic form". A \emph{locally Hamiltonian} \cite[Prop. 3.3.6]{Abraham78} vector field $V: M \times \bR \to TM$ is one that preserves $\omega$, that is, one for which
\beq{locallyH}
	\cL_V \omega = 0  ,
\eeq
where $\cL_V$ is the Lie derivative (see \Eq{LieDeriv} in \App{appendix}). 
Cartan's formula \Eq{LieIdentity} and the assumption that $d\omega = 0$ together give
\[
	\cL_V\omega = d(\imath_V\omega)  .
\]
In this case, \Eq{locallyH} implies that $d(\imath_V\omega) = 0$; in other words, whenever $V$ is locally Hamiltonian,  $\imath_V\omega$ is closed. 

If, in addition, this form is exact,
\beq{globallyH}
	\imath_V\omega = dH ,
\eeq
then $V$ is \emph{globally Hamiltonian}, with the Hamiltonian function $H: M \times \bR \to \bR$. By Darboux's theorem \cite[Thm 3.2.2]{Abraham78}, there is a neighborhood of each point in $M$ in which there are coordinates $(q,p)$ so that $\omega = dq \wedge dp$. 
In these coordinates, \Eq{globallyH} takes the form
\[
	\dot q = \partial_p H , \quad \dot p = - \partial_q H ,
\]
i.e., a canonical Hamiltonian system.

More generally, suppose an $n$-dimensional manifold $M$ is endowed with a nondegenerate form, $\Omega \in \Lambda^n(M)$, i.e., a ``volume-form". A vector field $V$ is incompressible with respect to $\Omega$, or \emph{locally Liouville}, if
\beq{locallyI}
	\cL_V \Omega \equiv (\nabla \cdot V) \Omega = 0 .
\eeq
Liouville's theorem then implies that the volume of any region is preserved by the flow of $V$ \cite[\S 9.2]{Meiss07}. As before, \Eq{locallyI} implies that $i_V\Omega$ is closed; if it is also exact, a global analog can be defined:
\begin{definition}[Globally Liouville {\cite[\S 2]{Marmo81}}] \label{def:Liouville}
  A vector field $V$ on a manifold $M$ with volume form $\Omega$ is \emph{globally Liouville} if 
  \beq{Liouville}
    \imath_V \Omega = d\beta,
  \eeq
  for some $\beta \in \Lambda^{n-2}(M)$.
\end{definition}
Of course, if $M$ has trivial cohomology then every closed form is exact, and there is no distinction between locally and globally Hamiltonian or Liouville vector fields. More generally, there may be some global obstruction to the existence of $H$ or $\beta$.  For example, suppose that $M = \bT^2$, and $\Omega = \omega = d\theta_1 \wedge d\theta_2$ is the volume/symplectic form.   
Then the vector field $V = \rho_i \partial_{\theta_i}$, for a constant rotation vector $\rho$, is incompressible because $i_V\Omega = \rho_1 d\theta_2 - \rho_2 d\theta_1$ is closed. However, since this form is not exact ($\theta_1$ and $\theta_2$ are not smooth functions on $M$), $V$ is not Hamiltonian, or equivalently, not Liouville.

As a second example, suppose $M = \bR^3$ and $\Omega$ is the standard volume, $\Omega = dx_1\wedge dx_2 \wedge dx_3$. Using the natural identification of a two-form $\zeta = \epsilon_{ijk} \zeta_i dx_j \wedge dx_k$ with the vector $\vec \zeta = \zeta_i \hat e_i$  and a one-form $\beta = \beta_i dx_i$ with a vector $\vec \beta = \beta_i \hat e_i$, \Eq{Liouville} reduces to the statement that $\vec V = \nabla \times \vec \beta$.  
For instance, if $V = \dot x_i \partial_{x_i}$ is a Beltrami vector field on $\bR^3$, i.e. $\vec V = \nabla \times \vec V$, then $\vec \beta = \vec V$, or  
\beq{BeltramiBeta}
\beta = \dot x_i dx_i.
\eeq
The ABC vector field (see \Sec{abc}) is Beltrami, and hence has this property.

The symplectic form $\omega$ is, by definition, closed. If it is also exact, then there is a one-form $\nu$ (often called the Liouville form) such that
$
	\omega = -d\nu
$.
Then if $V$ is globally Hamiltonian,
\[
	\cL_V \nu = \imath_V d\nu + d(\imath_V\nu) = d(\imath_V\nu - H) = dL ,
\]
where $L$ is the phase space Lagrangian. In canonical coordinates we could choose $\nu = p \cdot dq$, in which case $L = p \cdot \dot{q} - H$.

Similarly, if a volume form $\Omega$ is exact, we write
\beq{alphaDefine}
	\Omega = d\alpha  .
\eeq
Then, if $V$ is globally Liouville,
\[
	\cL_V \alpha = \imath_V d\alpha + d(\imath_V \alpha) = d(\imath_V\alpha + \beta).
\]
Here $\imath_V\alpha + \beta \in \Lambda^{n-2}(M)$ is the Liouvillian analog of the Lagrangian $L$, so we make the following definition:
\begin{definition}[Lagrangian Form] \label{def:LagrangianForm}
  Suppose that the volume form $\Omega = d\alpha$ is exact and the vector field $V$ is globally Liouville on a manifold $M$.  Then the \emph{Lagrangian form} $\lambda \in \Lambda^{n-2}(M)$ is defined by
  \beq{lambdaDefine}
    \cL_V \alpha = d\lambda \;, \mbox{ where } \lambda = \imath_V\alpha + \beta,
  \eeq
 and $\imath_V \Omega = d\beta$.
\end{definition}
For example, the standard volume $\Omega$ is exact on $M = \bR^3$ with
\beq{chosenAlpha}
	\alpha = x_3\, dx_1 \wedge dx_2.
\eeq
When $V = \dot x_i \partial_{x_i}$ is Beltrami,  \Eq{BeltramiBeta}, \Eq{lambdaDefine}, and \Eq{chosenAlpha} imply
\beq{BeltramiLambda}
	\lambda = \epsilon_{3jk} x_3 \dot x_j\, dx_k + \dot x_i\, dx_i.
\eeq

The discrete analogs of globally Liouville flows are exact volume-preserving maps; we define them here and note some key properties in anticipation of their use in \Sec{microdroplet}.
\begin{definition}[Exact Volume-Preserving Map \cite{Lomeli09}]
  Suppose that the volume form $\Omega = d\alpha$ is exact on a manifold $M$.  Then a diffeomorphism $R:M\to M$ is \emph{exact volume-preserving} if there exists a \emph{generating form} $\eta \in \Lambda^{n-2}(M)$ such that 
  \beq{exactVP}
    \alpha - R_*\alpha = d\eta.
  \eeq
\end{definition}
\noindent
Here we use the pushforward $R_*$, recall \Eq{pushForwardF}, instead of the pullback of \cite{Lomeli09} for later convenience. 

It is straightforward to show that if $R_1$ and $R_2$ are exact volume-preserving maps with generating forms $\eta_1$ and $\eta_2$, respectively, then the composition $R = R_1 \circ R_2$ is also exact volume-preserving \cite{Lomeli09} with generating form 
\beq{composition}
 \eta = \eta_1 + R_{1*} \eta_2.
\eeq
The generating form $\eta$ is the discrete analog of the Lagrangian form  \Eq{lambdaDefine}, and interacts with the latter as follows.

\begin{lemma} \label{lem:newLambda}
  Suppose that $\Omega = d\alpha$ is exact, $V$ is globally Liouville with Lagrangian form $\lambda_V$, and $R$ is exact volume-preserving with generating form $\eta$. Then the vector field $W = R_*V$ is globally Liouville with Lagrangian form
  \beq{newLambda}
    \lambda_W = R_*\lambda_V + \cL_W\eta.
  \eeq
\end{lemma}
\begin{proof}
  That $W$ is globally Liouville follows directly from \Eq{naturally} and the invariance of $\Omega$ under $R_*$:
  \[
    \imath_W \Omega = R_*(\imath_V \Omega) = d(R_*\beta_V) := d\beta_W,
  \]
  where $\imath_V \Omega = d\beta_V$.  Using \Eq{LieDeriv} and \Def{LagrangianForm}, the Lagrangian form for $W$ is derived by
\[    
	\cL_W \alpha = \cL_W( R_* \alpha + d\eta ) 
                = R_*( \cL_V \alpha ) + \cL_W( d\eta ) 
                = d( R_* \lambda_V + \cL_W \eta ),
\]
which gives \Eq{newLambda}.
\end{proof}

As we will see in \Sec{lobeVolumes}, the form $\lambda$ plays a central role in the computation of the volumes of lobes formed by the intersection of past- and future-invariant regions; it is analogous to the phase space Lagrangian we used to compute such volumes for the 2D case \cite{Mosovsky11}.

\section{Action-Flux Formulas for Lobe Volumes} \label{sec:lobeVolumes}

In this section we will obtain the action-flux formulas to compute the transport fluxes. As a standing assumption, $V$ will denote a transitory vector field \Eq{transitoryODE} that is globally Liouville with respect to an exact volume form $\Omega$ and that has a complete flow $\vphi_{t,t_0}$.

Suppose that $\cP_0,\, \cF_\tau \subseteq M$ are past- and future-invariant regions, respectively.  By definition, trajectories within $\cP_0$ at $t=0$ will remain within it for all $t < 0$: $\cP_0$ is coherent under $P$ in the Lagrangian sense. Similarly, $\cF_\tau$ is coherent under the future vector field $F$.  As a result, any transport between $\cP_0$ and $\cF_\tau$ must occur during the transition interval $[0,\tau]$, and the transported phase space itself is the collection of regions $\cR_t = \cP_t \cap \cF_t$, i.e., the intersection of $\cP$ and $\cF$ in any slice. We will call the components of the slices $\cR_t$, ``lobes.''  Since $V$ is Liouville, the intersection volume, or \emph{flux} from $\cP_0$ to $\cF_\tau$,
\beq{flux}
  \Phi = \Vol(\cP_t \cap \cF_t),
\eeq
is independent of time. In particular, using the transition map \Eq{transitionMap}, $\Phi = \Vol(T(\cP_0) \cap \cF_\tau)$.

Let $\cU = \partial \cP$ and $\cS = \partial \cF$ be the boundaries of the orbits of the future- and past-invariant regions so that the lobes, $\cR_t$, are bounded by pieces of the slices $\cU_t$ and $\cS_t$. As we will see below a key set in the action-flux formulas will be $\cI = \cU \cap \cS$, the set of orbits at the intersection of the lobe boundary components. 

In some cases $\cU$ and $\cS$ will be pieces of stable and unstable manifolds of future- and past-hyperbolic sets. For example, suppose that $\cP_0$ and $\cF_\tau$ are topological balls whose boundaries are portions of the closures of the codimension-one unstable and stable manifolds of past- and future-hyperbolic equilibria, $p$ and $f$, respectively.\footnote
{For example, $p$ could be the equilibrium $p^1$ of the microdroplet flow of \Sec{microdroplet}, see \Fig{contours}(a), and thus $\cP_0$ is the droplet itself.}
Under the transitory flow, portions $\cU_t \subset \Wu_t(p)$ and $\cS_t \subset \Ws_t(f)$ may intersect to bound a lobe $\cR_t$, as sketched in \Fig{simpleLobe}(a), which is also a ball. In this case, the intersection set, 
\beq{intersectionCurve}
  \cI_t = \cU_t \cap \cS_t ,
\eeq
is an ($n-2$)-dimensional sphere.  It is clear that $\Wu_t(p)$ is past invariant and $\Ws_t(f)$ is future invariant (though the lobe boundary surfaces $\cU_t$ and $\cS_t$  themselves are not), and, in this case, that $\cI_t \to p_0$ as $t \to -\infty$ and $\cI_t \to f_\tau$ as $t \to \infty$. 

\InsertFig{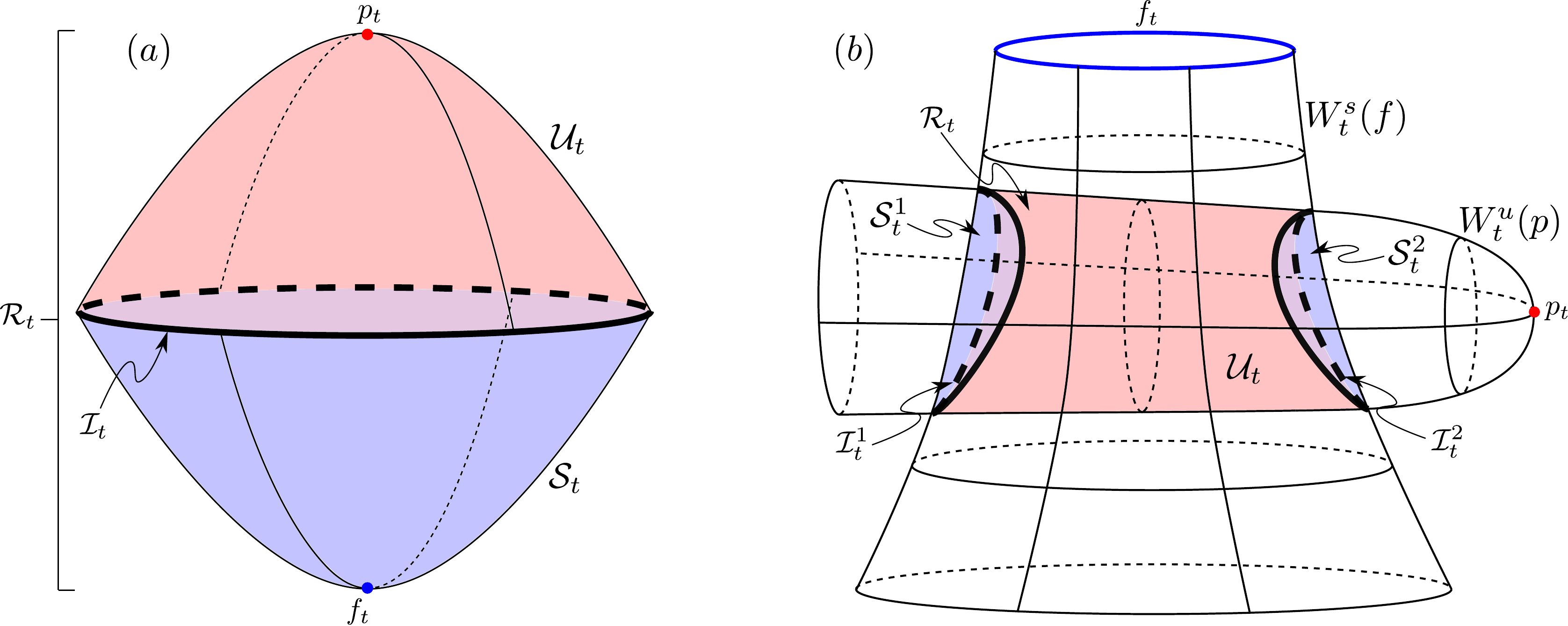}{(a) Simple lobe $\cR_t \subset M$ formed by the intersection of $\cU_t \subset \Wu_t(p)$ and $\cS_t \subset \Ws_t(f)$.  (b) Similar to (a), except there are three surfaces that make up the lobe boundary: $\cS_t^1,\, \cS_t^2 \subset \Ws_t(f)$ and $\cU_t \subset \Wu_t(p)$.  In both (a) and (b), intersections of the manifolds are shown as bold black curves and only the lobe boundaries $\partial \cR_t$ are shaded.}{simpleLobe}{.9\linewidth}

However, $\cI_t$ need not be connected, $p$ and $f$ need not be single orbits, and $\cR_t$ need not be a topological ball. For example, in \Fig{simpleLobe}(b), $f_t$ now represents a loop whose orbit is a future-hyperbolic periodic orbit, while $p$ remains a past-hyperbolic equilibrium. In this case the lobe boundary contains two disjoint intersection curves, $\cI_t^1$ and $\cI_t^2$, both of which contract to $p_0$ in the past, but in forward time approach the periodic orbit $f_\tau$.  As will become apparent, to compute the volume of $\cR_t$ using the action-flux formulas, it is convenient (though not necessary) that the surface areas of $\cU_t$ and $\cS_t$ converge to zero in the appropriate limit in time.

Since $V$ is globally Liouville, the flux \Eq{flux} is independent of time.  It is this flux that we wish to compute, as it represents the portion of the past-invariant region $\cP_0$ transported to the future-invariant region $\cF_\tau$. Stokes's theorem, using \Eq{alphaDefine}, allows for an immediate reduction of the volume integral for $\cR_t$ to an integral over its boundary:
\beq{StokesTheorem}
  \Phi = \textrm{Vol}(\cR_t) = \int_{\cR_t} \Omega = \int_{\partial \cR_t} \alpha .
\eeq
For the lobe $\cR_t$ in \Fig{simpleLobe}(a), $\partial \cR_t = \cU_t + \cS_t$, while in \Fig{simpleLobe}(b) there are two disjoint surfaces $\cS^1_t,\, \cS^2_t \subset \Ws_t(f)$ on the boundary, so that $\partial \cR_t = \cU_t + \cS^1_t + \cS^2_t$.  In general, $\partial \cR_t$ can be decomposed into pieces
$\cU^j_t \subset \partial \cP_t$ and $\cS^i_t \subset \partial \cF_t$ of the boundaries of the regions $\cP_t$ and $\cF_t$.  Thus \Eq{StokesTheorem} becomes a sum of integrals of $\alpha$ over such submanifolds. It is to this computation that we now turn.

Evaluation of \Eq{StokesTheorem} could be performed by numerical evaluation of the $(n-1)$-dimensional surface integrals; however, this requires an accurate representation of the surfaces $\cU_t,\cS_t \subset \partial \cR_t$, which also implicitly requires knowing their time evolutions.  This additional temporal information is essentially ``wasted'' since it is not explicitly used to compute the surface integrals.  Furthermore, the exponential stretching typical of chaos can make obtaining well-resolved representations of $\cU_t$ and $\cS_t$ computationally prohibitive.

Our alternative reduces the dimension of the Lagrangian information necessary to evaluate \Eq{StokesTheorem} by computing the ``generalized actions'' of the orbits on the boundary $\cI_t$, and requires evaluating an $(n-2)$-dimensional spatial integral plus a temporal integral.  In the extended phase space, this corresponds to an $(n-1)$-dimensional integral, but it explicitly uses the time evolutions of $\partial \cU_t$ and $\partial \cS_t$ that must be computed in any case.  Our formulation applies to the case where $\partial R_t$ has components that are not stable or unstable manifolds of any future- or past-hyperbolic orbit (e.g., see \Sec{microdroplet}).  Indeed, since the orbit of \emph{any} subset of $M$ is invariant under $\vphi$ in $M \times \bR$, the result applies to general codimension-one submanifolds $\Gamma_t \subset M$.

\begin{theorem}\label{thm:mainResult}
 Suppose that $\Omega = d\alpha$ is an exact volume form on $M$ and $\Gamma_t$ is a codimension-one slice of an invariant set of the flow $\vphi$ of a globally Liouville vector field $V$.  Then for any $r \in \bR$,
 \beq{mainResult}
   \int_{\Gamma_t} \alpha = \int_r^t \left( \int_{\partial \Gamma_s} \lambda \right) ds + \int_{\Gamma_r} \alpha.
 \eeq
\end{theorem}
\begin{proof}
  Differentiating $\alpha$ along the vector field $V$, employing Cartan's homotopy formula \Eq{LieIdentity} (see \App{appendix}), and using \Eq{Liouville} gives
  \[
    \frac{d}{dt} \alpha = \cL_V \alpha = d( \imath_V \alpha + \beta ) = d\lambda.
  \]
Integrating this expression from $r$ to $t$ using \Eq{naturally} then results in 
  \[
   \alpha - \vphi_{r,t}^*\alpha = \int_r^t \frac{d}{ds} \vphi_{s,t}^* \alpha \, ds = \int_r^t d( \vphi_{s,t}^*\lambda )\, ds,
  \]
  for any $r$.  A second integration over $\Gamma_t$ and rearrangement then gives
\begin{align*}
   \int_{\Gamma_t} \alpha &= \int_r^t \left( \int_{\Gamma_t} 
         d(\vphi_{s,t}^* \lambda) \right) ds + 
                             \int_{\Gamma_t} \vphi_{r,t}^* \alpha \notag \\
                          &= \int_r^t \left( \int_{\Gamma_s} d\lambda \right) ds 
                            + \int_{\Gamma_r} \alpha ,\notag
\end{align*}
which immediately reduces to \Eq{mainResult} using Stokes's theorem.
\end{proof}

It is interesting to note that the flow implicit in the time-integration in \Eq{mainResult} may be chosen independently of the original transitory flow, provided it is globally Liouville.  This observation will be used to simplify the computations in \S\ref{sec:abc}-\ref{sec:microdroplet}, and an example of its implementation is discussed in \App{SpeedUp}.

Equation (\ref{eq:mainResult}) simplifies if the surface area of $\Gamma_t$ limits to zero in either backward or forward time.  This typically occurs when $\Gamma_t$ is a compact subset of an invariant manifold of a past- or future-hyperbolic orbit (e.g., as in \Fig{simpleLobe}).  Under these assumptions, one can take the limit $r \to \pm \infty$ in \bTh{mainResult} to obtain the following.
\begin{corollary} \label{cor:corollary}
  Under the assumptions of \bTh{mainResult}, if the $\alpha$-surface area of $\Gamma_t$ vanishes as $t \to -\infty$,
\beq{pastAction}
    \int_{\Gamma_\tau} \alpha = \int_{-\infty}^\tau 
    	\left( \int_{\partial \Gamma_s} \lambda \right) ds  ,
\eeq
and if the $\alpha$-surface area of $\Gamma_t$ vanishes as $t \to +\infty$,
\beq{futureAction}
    \int_{\Gamma_\tau} \alpha = -\int_\tau^\infty 
    		\left( \int_{\partial \Gamma_s} \lambda \right) ds  .
\eeq
\end{corollary}


We will refer to (\ref{eq:mainResult})--(\ref{eq:futureAction}) as the \emph{action-flux} formulas.  Since $\lambda$ is the $n$-dimensional analog of the Lagrangian, its integral along a codimension-two set of orbits gives a \emph{generalized action} for that set.  Thus, the action-flux formulas, in conjunction with \Eq{StokesTheorem}, allow us to calculate the flux by computing the generalized action of sets of key orbits on the lobe boundary.  

For example, suppose that a lobe boundary $\partial \cR_\tau$ can be decomposed into $N_{\pm}$ connected submanifolds $\Gamma^{i\pm}_\tau$ that collapse in forward or backward time, respectively.  Then \Eq{StokesTheorem}, with \Eq{pastAction} and \Eq{futureAction}, yields
\beq{volume}
  \Phi = 
       \sum_{j=1}^{N_-} \int_{-\infty}^\tau \left( \int_{\partial \Gamma^{j-}_s} \lambda \right) ds
       - \sum_{i=1}^{N_+} \int_\tau^\infty \left( \int_{\partial \Gamma^{i+}_s} \lambda \right) ds .
\eeq
Implicit in \Eq{volume} is a choice of orientation on the boundaries.  We will always orient $\cR_t$ with respect to a right-handed outward normal; this induces orientations on $\Gamma^{j-}_t$ and $\Gamma^{i+}_t$, and, in turn, on their boundaries $\partial\Gamma^{j-}_t$ and $\partial\Gamma^{i+}_t$.  

\section{Example: Transitory ABC Flow} \label{sec:abc}

The ABC vector field \cite{Arnold65},
\beq{ABCVector} 
  \dot{\mathbf{x}} = 
  \left( \begin{matrix}
          A \sin z + C \cos y \\
          B \sin x + A \cos z \\
          C \sin y + B \cos x
  \end{matrix} \right),
\eeq
models a steady, inviscid, incompressible Beltrami flow on $\bT^3$. 
Interestingly, it is an exact solution to the Navier-Stokes equations, provided an appropriate forcing term is added to counter the effects of viscous dissipation; moreover for small Reynolds numbers, it is stable \cite{Galloway87}.  Despite the steadiness of the flow, its streamlines are chaotic \cite{Henon66, Dombre86, Zhao93, Huang98}; hence \Eq{ABCVector} is a prototypical example of a laminar vector field with complicated Lagrangian dynamics. 

The ABC vector field is locally Liouville; however, since there is no exact volume form on $\bT^3$, it is not globally Liouville on this manifold.  In order to apply the action-flux formulas of \Sec{lobeVolumes}, we lift the $z$ coordinate to $\bR$, letting the phase space become $M = \bT^2 \times \bR$. In this case, the standard volume $\Omega = dx \wedge dy \wedge dz$ is exact on $M$, and we take $\alpha = z\, dx \wedge dy$.  With this choice,  \Eq{ABCVector} is Beltrami, and \Eq{BeltramiBeta} and \Eq{BeltramiLambda} give
\beq{abcLambda} \begin{split}
          \beta &= (A \sin z +C\cos y)dx + (B \sin x + A\cos z)dy + (C\sin y + B\cos x)dz, \\
          \lambda  &= (A \sin z +C\cos y)(dx+zdy) + (B \sin x + A\cos z)(dy -zdx) \\
	           &\hskip.3in + (C\sin y + B\cos x)dz.
\end{split} \eeq

Several recent studies have used finite-time Lyapunov exponents to analyze both steady and unsteady generalizations of \Eq{ABCVector} \cite{Haller01, Sadlo09}; these have focused primarily on extracting Lagrangian coherent structures by determining regions that experience maximal local stretching.  Here, we study a transitory ABC flow in which the identification of coherent structures of $P$ and $F$ is trivial, and focus on computing the flux between these structures.

\subsection{Transitory System} \label{sec:transitoryABC}
Modulating the coefficients $A$, $B$, and $C$ of \Eq{ABCVector} over a compact temporal interval makes it transitory in the sense of \Eq{transitoryODE}. We choose to set $C=0$ for $t<0$ and $B=0$ for $t>\tau$, to give past and future vector fields
\beq{pastFuture}
  P(\bx) = \left( \begin{array}{c}
                    A \sin z \\
                    B \sin x + A \cos z \\
                    B \cos x
                  \end{array} \right),
  \qquad 
  F(\bx) = \left( \begin{array}{c}
                    A \sin z + C \cos y \\
                    A \cos z \\
                    C \sin y
                  \end{array} \right ). \\
\eeq
The full transitory vector field is then given by the convex combination
\beq{transitoryABC}
  V(\bx,t) = \left( \begin{array}{c}
                        A \sin z + s(t)C \cos y \\
                        (1-s(t))B \sin x + A \cos z \\
                        s(t)C \sin y + (1-s(t))B \cos x
                      \end{array} \right),
\eeq
as in \Eq{transitoryModel}, and we use the cubic transition function \Eq{cubicTransition}.
We will denote the flows of $P$ and $F$ by $\vphi^P$ and $\vphi^F$, respectively, and the flow of the full transitory vector field \Eq{transitoryABC} by $\vphi$.  

The autonomous vector fields $P$ and $F$ are integrable \cite{Dombre86}; indeed, they have invariants
\beq{levelSetFunc} \begin{split}
  H_P(\bx) &= B \sin x + A \cos z, \\
  H_F(\bx) &= A \sin z + C \cos y,
\end{split} \eeq
respectively. Moreover, these functions act as Hamiltonians that generate the flows of $P$ in the $(x,z)$ plane and $F$ in the $(y,z)$ plane. Since $\dot y = H_P$ for $P$ and $\dot x = H_F$ for $F$, the motion in these transverse directions is trivial; consequently, the flows of $P$ and $F$ can be completely characterized by their two-dimensional portraits, see \Fig{streamlines}. Note that the level sets of the invariants \Eq{levelSetFunc} become two-tori in $M$, with the exception of certain critical sets, which correspond to separatrices. 

\InsertFigTwo{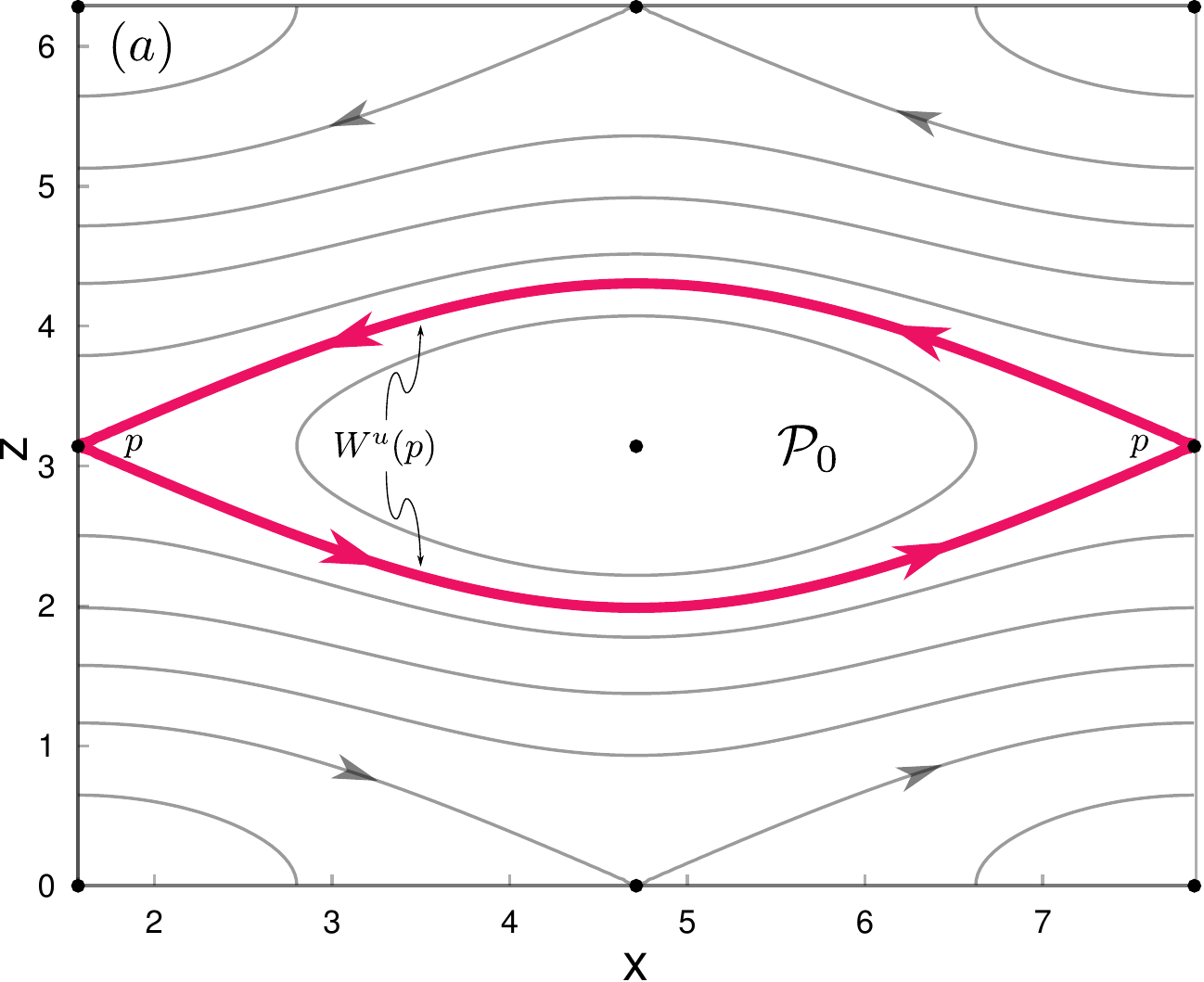}{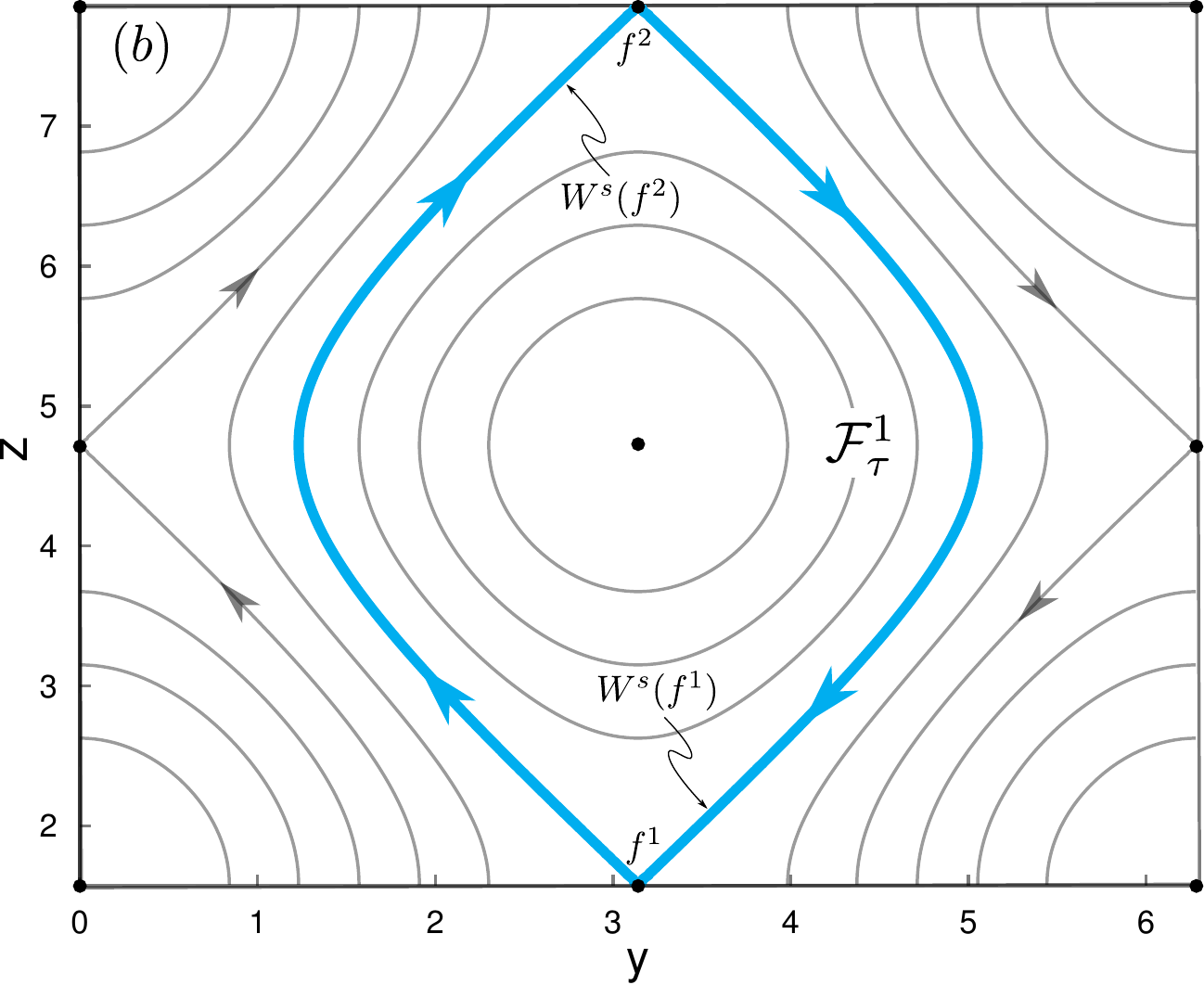}{Poincar\'e sections of the invariant two-tori of (a) $P$ and (b) $F$ for $(A,B,C)=(1,0.3,1.5)$. The invariant manifolds of the equilibrium $p$ of $P$ are the bold (red) curves in (a), and those of the equilibria  $f^k$ of $F$ are the bold (blue) curves in (b). The ranges of $(x,y,z)$ are shifted to visualize the resonance zones $\cP_0$ and $\cF_\tau^1$.}{streamlines}{.45\linewidth}

In each $2\pi$ range of $z$, $P$ and $F$ generally have two elliptic and two hyperbolic periodic orbits; in the special case $A=B=C=1$, these become lines of fixed points.  
We will set $A=1$, without loss of generality, and assume that
\beq{coeffRestriction}
  0 < B < A = 1 < C;
\eeq
with this choice, the resulting phase portraits are like those in \Fig{streamlines}. The past- and future-invariant regions that we will analyze are bounded by the manifolds of the hyperbolic periodic orbit
\beq{p} \notag
  p = \{(\tfrac{\pi}{2}, y, \pi) \:\: \big| \:\: y \in [0,2\pi] \}
\eeq
of $P$ and the hyperbolic periodic orbits
\beq{fk}
  f^k = \{(x, \pi, \tfrac{\pi}{2} + 2\pi (k-1)) \:\: \big| \:\: x \in [0,2\pi]\} ,
\eeq
of $F$ (two of these are shown in \Fig{streamlines}(b)).

The invariant manifolds---under $P$---of $p$ correspond to the level set
\beq{HPstar}
	H_P(\bx) = H_P^* := B-A .
\eeq
Because of the horizontal periodicity of $M$, $\Wu_0(p)$ forms a pair of homoclinic connections, shown in cross-section in \Fig{streamlines}(a). Each is homeomorphic to a two-torus in $M$, as shown in \Fig{Intersection3D_1}, and together, these manifolds bound a past-invariant region $\cP_0$. Similarly, the invariant manifolds---under $F$---of $f^k$ correspond to the level set
\beq{HFstar}
	H_F(\bx) = H^*_F := A - C .
\eeq
Since $M$ is unbounded in $z$, these form heteroclinic connections, see \Fig{streamlines}(b) and \Fig{Intersection3D_1}(a). The heteroclinic connections between $f^{k}$ and $f^{k+1}$ bound a future-invariant region $\cF_\tau^k$; for each $k \in \bZ$ these are just shifted copies of $F_{\tau}^1$, shown in \Fig{Intersection3D_1}.   

\InsertFig{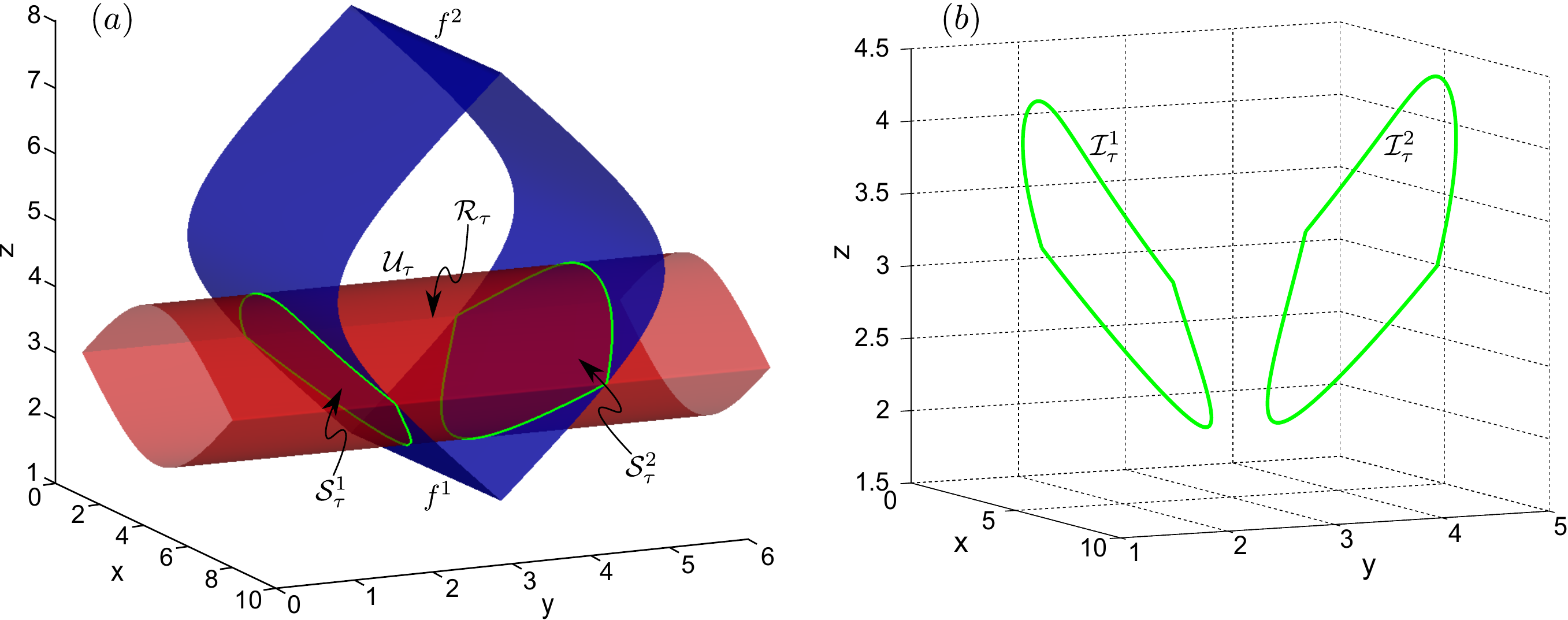}
{(a) The past- and future-invariant regions for \Eq{transitoryABC}
with $(A,B,C) = (1.0, 0.3, 1.5)$ are bounded by the invariant manifolds $\Wu(p)$ (red) and $\Ws(f^{1,2})$ (blue), respectively. The lobe shown corresponds to $\cR_\tau = \cP_\tau \cap \cF_\tau^1$ for $\tau = 0$. Its boundary consists of three surfaces $\cU_\tau \subset \Wu(p)$, $\cS_\tau^1 \subset \Ws(f^2)$, and $\cS_\tau^2 \subset \Ws(f^1)$, which intersect in two loops $\cI_\tau^k$ (green) shown in (b).}{Intersection3D_1}{\linewidth}

It is important to remember that $\cP_0$ and $\cF_\tau^k$ are \emph{not} invariant under the transitory flow $\vphi$. However, they are Lagrangian coherent structures, or ``resonance zones," of the past and future vector fields, respectively.
We consider the problem of computing the transport from $\cP_0$ to each of the $\cF_{\tau}^k$; that is, the volume of the lobes $\cR_\tau^k = \cP_\tau \cap \cF_\tau^k$. 
There is always at least one such lobe for any $\tau$ and choice of parameters subject to \Eq{coeffRestriction}, and when $\tau$ is finite there are only finitely many.  The accompanying movie file (\texttt{ABC Flow Lobes Vs.~Transition Time}) shows the lobes at $t=\tau$, for increasing values of $\tau$ and for the same parameters as in \Fig{Intersection3D_1}.

\subsection{Computation} \label{sec:computation}

The volumes of the lobes $\cR^k_\tau$ will be computed using the action-flux formulas \Eq{pastAction} and \Eq{futureAction}, which rely on knowing the orbits of the intersection curves
\[
  \cI_\tau = \Wu_\tau(p) \cap  \bigcup_{k \in \bZ} \Ws_\tau(f^k).
\]
As an example, two such curves, $\cI^1_\tau$ and $\cI^2_\tau$, are shown in \Fig{Intersection3D_1}(b).  For the moderate values of $\tau$ that we study below, only the intersections of $\cP_\tau$ with $\cF^1_\tau$ and $\cF^0_\tau$ are nonempty, and thus the only lobes formed are $\cR^1_\tau$ (the \emph{primary} lobe) and $\cR^0_\tau$ (the \emph{secondary} lobe).  To simplify notation, we adopt the convention of referring to elements of the secondary lobe with a ``tilde'' (i.e. $\tilde \cR_\tau$, $\tilde \cI_\tau$, etc.), and omit the superscripts for both lobes (cf.~\Fig{twoLobes}).

The computation of the intersection curves is done with a root finding and continuation method, and is simplified by using the level sets \Eq{HPstar} and \Eq{HFstar}. It is convenient to parameterize the past manifold as $G: \bT^2 \to \Wu_0(p) = \partial \cP_0$, and search for intersections in parameter space. Using $(u,v)$ as the parameters, \Eq{HPstar} gives $ G(u,v) = ( x(v), u, z(v) )$, where
\bsplit{parameterization}
   x(v) &= \left\{ \begin{array}{ll}
                   2v + \tfrac{\pi}{2}, &v \in [0,\pi), \medskip \\
                   \tfrac{9\pi}{2} - 2v, &v \in [\pi,2\pi],
                 \end{array} \right. \\
  z(v) &= \left\{ \begin{array}{ll}
                      \cos^{-1}\left( 2\frac{B}{A} \sin^2 v-1 \right), 
                         &v \in [0,\pi), \medskip \\
                      2\pi - \cos^{-1}\left(2\frac{B}{A}\sin^2 v-1 \right), 
                        &v \in [\pi,2\pi],
                    \end{array} \right.
\esplit
see the sketch in \Fig{param}. Note that increasing $v$ corresponds to a counterclockwise circuit around the separatrix loop in the $xz$-plane, while $u$ is simply the $y$-coordinate.

\InsertFig{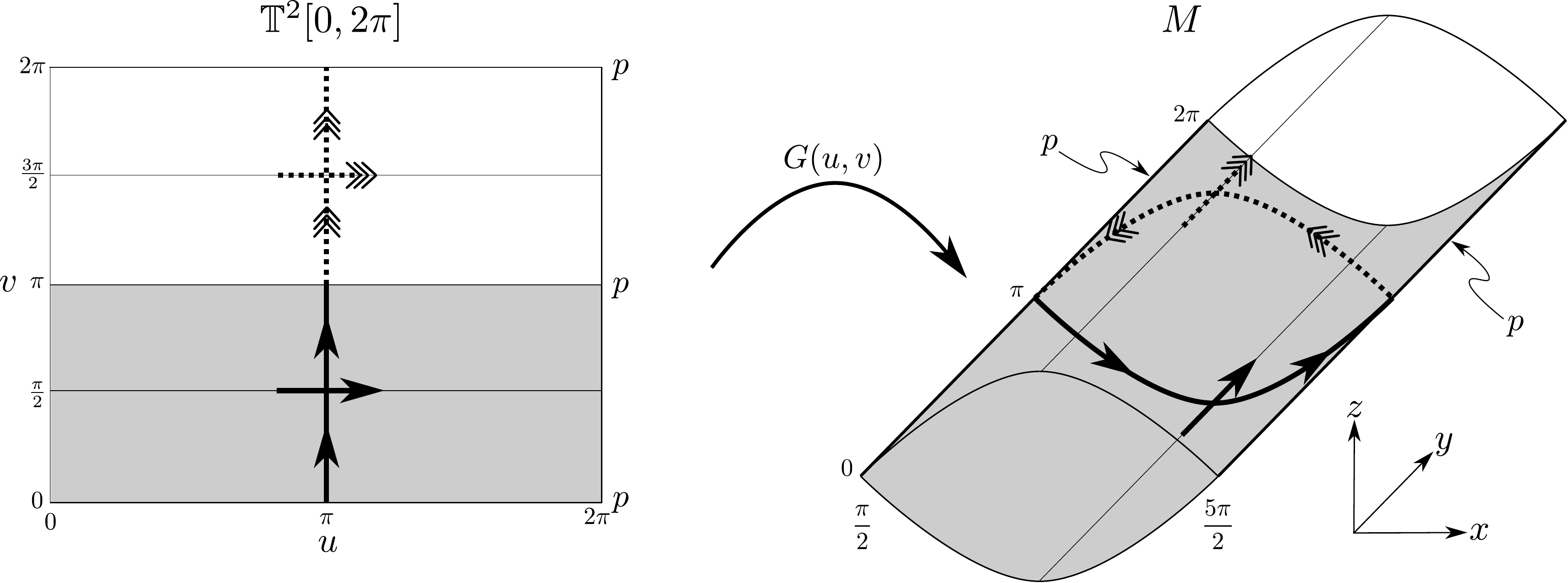}{Parameterization of $\Wu_0(p)$.  Arrows along the manifold are included to indicate orientation.}{param}{\linewidth}

The parametric representation of the intersection curves is then given in terms of the level set \Eq{HFstar} by
\[
  \cI_G = \big\{ (u,v) \in \bT^2 \:\: \big| \:\: H_F(T(G(u,v))) = H_F^* \big\},
\]
where $T$ is the transition map \Eq{transitionMap}.  Two examples are shown in \Fig{Intersection2D_1}; see \App{SpeedUp} for a discussion of the continuation techniques used in the computation of these curves. The corresponding intersections in phase space then become
$
  \cI_t = \vphi_{t,0}(G(\cI_G)),
$
and the curves shown in \Fig{Intersection3D_1}(b) are the phase-space representation
of those in \Fig{Intersection2D_1}(a) under this mapping, with $t = \tau = 0$.

\InsertFig{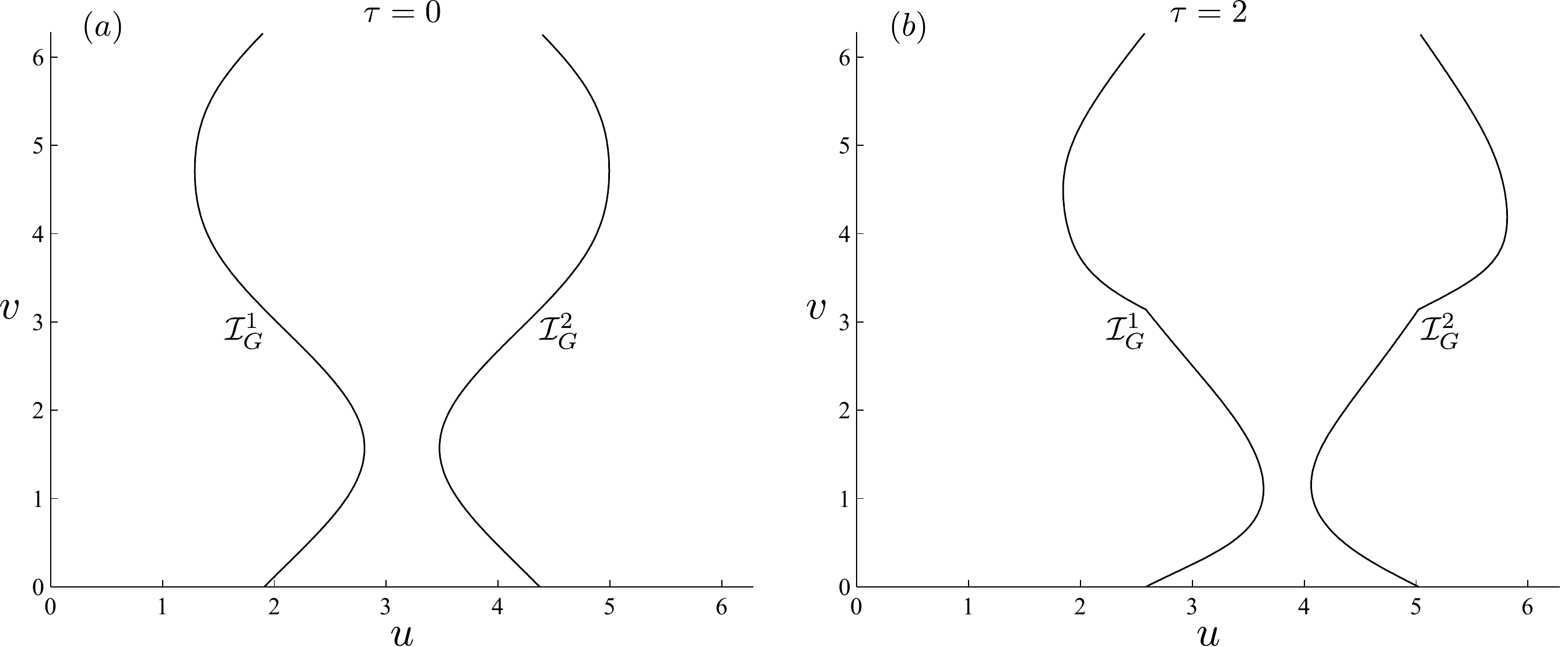}{Intersection curves $\cI_G$ in parameter space for $(A,B,C) = (1, 0.3, 1.5)$ with $\tau=0$ (a) and $\tau=2.0$ (b).  The accompanying movie file (\texttt{ABC Flow Manifold Intersections Vs.~Transition Time}) shows the dependence of $\cI_G$ on transition time, for $\tau \in [0.5, 7.5]$. }{Intersection2D_1}{\linewidth}

Given numerically-computed curves $\cI_G$, it is straightforward to compute the flux $\Phi$ using the action-flux formulas. We first discuss the case where there is one lobe, the primary lobe $\cR_\tau = \cP_\tau \cap \cF_\tau^1$, and its boundary has the form $\partial \cR_\tau = \cU_\tau \cup \cS_\tau^1 \cup \cS_\tau^2$ as in \Fig{Intersection3D_1}(a). This occurs when $B<\tfrac12$ and $\tau$ is small enough. The flux is then given by
\beq{simplestABCLobe}
  \Phi = \textrm{Vol}(\cR_\tau) = \int_{\cU_\tau} \alpha + \int_{\cS^1_\tau} \alpha + \int_{\cS^2_\tau} \alpha.
\eeq
Note that under the flow $\vphi$, $\cS^1_t \to f^2$ and $\cS^2_t \to f^1$ as $t \to +\infty$, and moreover, since these surfaces lie on the local manifolds $\Ws_t(f^k)$, their surface areas limit to zero.  Thus, by \Eq{futureAction}, 
\beq{stableIntegral}
  \int_{\cS_\tau^{j}} \alpha = -\int_{\tau}^\infty \left( \int_{\cI_s^{j}} \lambda \right) ds,
\eeq
where $\cI_\tau^j = \partial \cS_\tau^j$, and the orientations of these contours are chosen consistent with a right-handed outward normal to the lobe.
Further note that over the range of integration in \Eq{stableIntegral}, $\vphi = \vphi^F$.  It is also possible, and may be more efficient, to compute these integrals \emph{backward} in time under $\vphi^F$ (see \Eq{stableIntegralAlt} in \App{SpeedUp}).

To compute the integral over $\cU_\tau$ in \Eq{simplestABCLobe} we first use \Eq{mainResult} to pull the integration back to $t = 0$:
\beq{transformation}
  \int_{\cU_\tau} \alpha = \int_0^\tau \left( \int_{\partial \cU_s} \lambda \right) ds + \int_{\cU_0} \alpha.
\eeq
The first integral on the right-hand side is easily computed numerically, as it is over the compact transition interval. Computation of the second is a bit more difficult: since $\cU_0$ encircles $\Wu_0(p)$, it does not collapse to $p$ under $\vphi^P$, and consequently, its $\alpha$-surface area does not vanish in either direction of time. To get around this, we can divide $\cU_0$ into subsurfaces that do collapse under $\vphi^P$, reducing the last integral in \Eq{transformation} to a sum of integrals over these subsurfaces, each of which can be evaluated using \Eq{pastAction}--\Eq{futureAction} (see \Eq{unstableIntegral} and technical remarks on this splitting in \App{SpeedUp}).  With these techniques, the flux \Eq{simplestABCLobe} can be computed using \Eq{stableIntegral} and \Eq{transformation}. Several additional techniques can be used to speed up the computations and decrease numerical errors, see \App{SpeedUp}.

If $B > \tfrac12$, the past-invariant set $\cP_\tau$ intersects both $\cF^1_\tau$ and $\cF^0_\tau$ to form a primary lobe $\cR_\tau$ and a secondary lobe $\tilde \cR_\tau$, even when $\tau = 0$, see \Fig{twoLobes}.  
\InsertFig{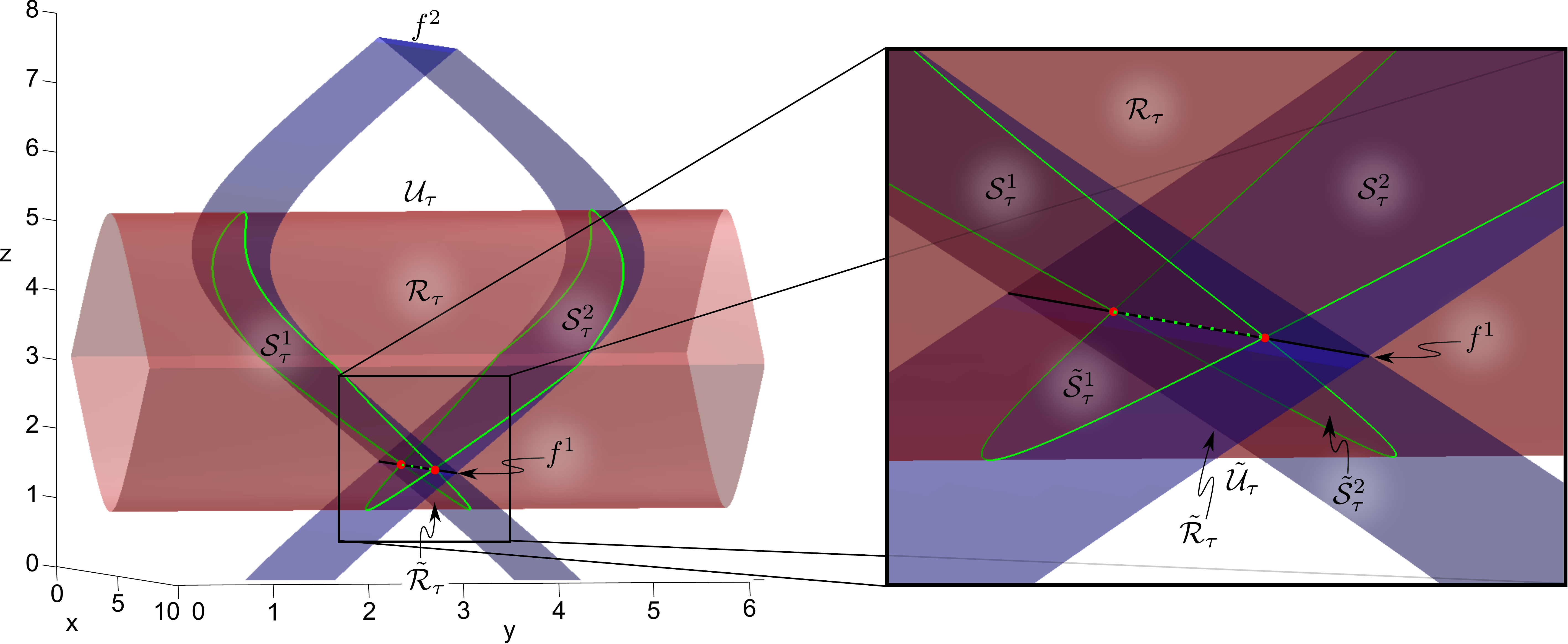}{Lobes at $t=\tau$ for $(A,B,C,\tau) = (1, 0.8, 1.5, 0)$.  The dotted green segment is added to $\partial\cS^{1,2}_\tau$ and $\partial\tilde\cS^{1,2}_\tau$ to ensure that these are closed curves.}{twoLobes}{\linewidth}
A secondary lobe also exists when $B < \tfrac12$, provided $\tau$ is large enough; for example, for the parameters used in \Fig{Intersection2D_1}, $\tilde \cR_\tau$ is formed at $\tau \approx 4.5$.  This lobe can be seen in \Fig{Intersection3D_2} for $\tau = 6$.
\InsertFig{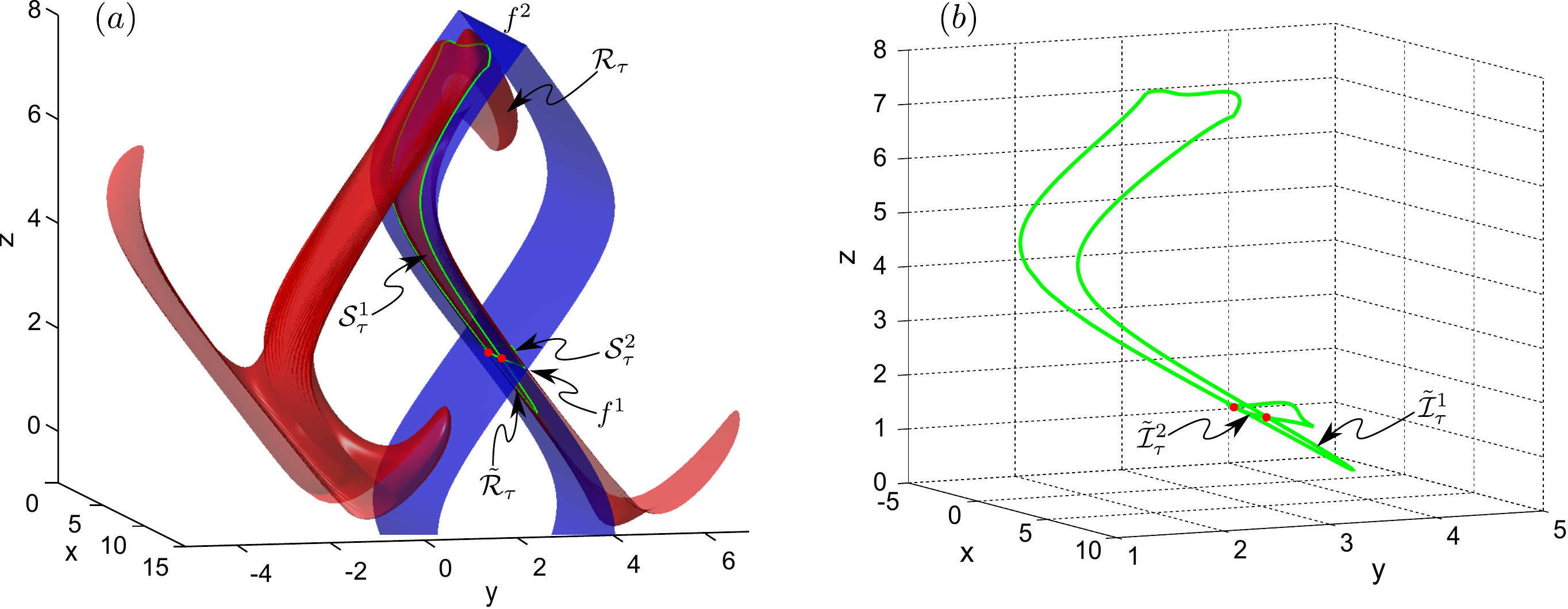}{Intersection of $\cP_\tau$ with $\cF_\tau$ for \Eq{ABCVector} with  $(A,B,C) = (1.0, 0.3, 1.5)$ and $\tau = 6$.  Intersection curves are shown (a) with their corresponding invariant manifolds $\Wu_\tau(p)$ (red) and $\Ws_\tau(f^{0,1,2})$ (blue), and (b) without them.}{Intersection3D_2}{\linewidth}
The total flux $\Phi$ is then simply the sum of the volumes of the two lobes, and computing each is similar to the single-lobe case described above.

\subsection{Results} \label{sec:abcResults}
Using the techniques discussed above, we computed $\Phi$ as a function of the parameters $B$ and $\tau$. The volumes of the primary and secondary lobes, each of which contributes to $\Phi$, were computed independently.  A summary of these lobe volumes as percentages of the volume of the past-invariant set $\cP_0$ is shown in \Fig{lobeVolSummary}.  The volume of $\cP_0$ itself is given by the quadrature\footnote
{Similarly, the volume of $\cF^k_\tau$ is given by the same formula upon replacing $\tfrac B A$ with $\tfrac A C$ and $x$ with $z$.}
\beq{analyticalVolume}
  \Vol(\cP_0) = 4\pi \left[ 2\pi^2 - \int_{\frac{\pi}{2}}^{\frac{5\pi}{2}} \cos^{-1} \big( \tfrac B A(1-\sin x) - 1 \big)\, dx \right].
\eeq
The curves in the figure denote the parameters at which the second lobe $\tilde \cR_\tau$ emerges.  Thus for $B$ values below this curve, the volume of the secondary lobe in \Fig{lobeVolSummary}(b) is zero. Note that for the parameter ranges of the figure, the flux due to the secondary lobe is never more than 3.5\% of the volume of $\cP_0$.  Also, the percent flux of the primary lobe does not strongly depend upon $\tau$; this variation is shown in \Fig{timeDependence} for $B=0.5$.

\InsertFigTwo{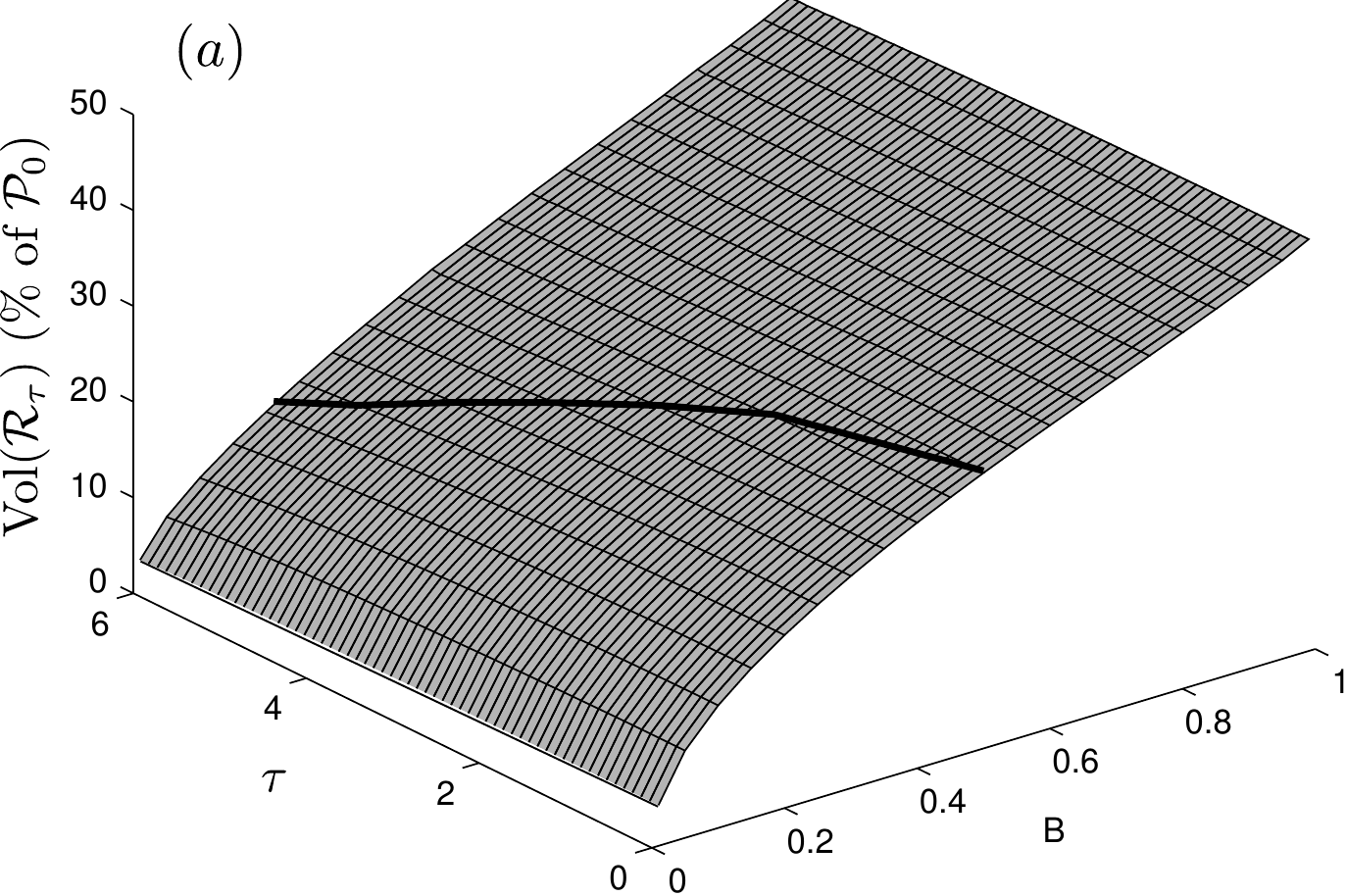}{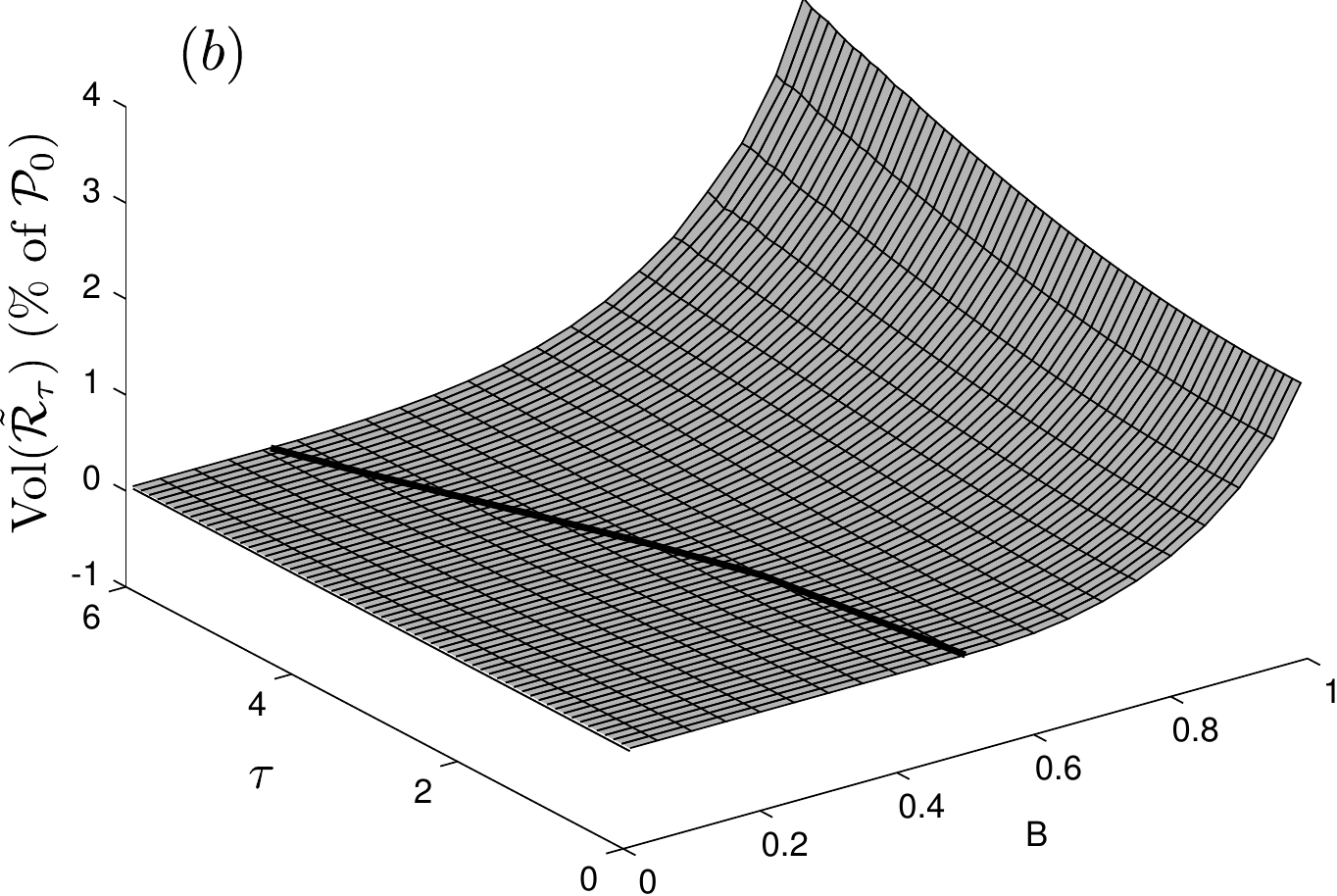}{Ratio of the volumes of the (a) primary and (b) secondary lobes to $\Vol(\cP_0)$.  Here $(A, C) = (1.0, 1.5)$, and $B$ and $\tau$ vary.  The  curves denote emergence of the secondary lobe.}{lobeVolSummary}{.45\linewidth}

\InsertFig{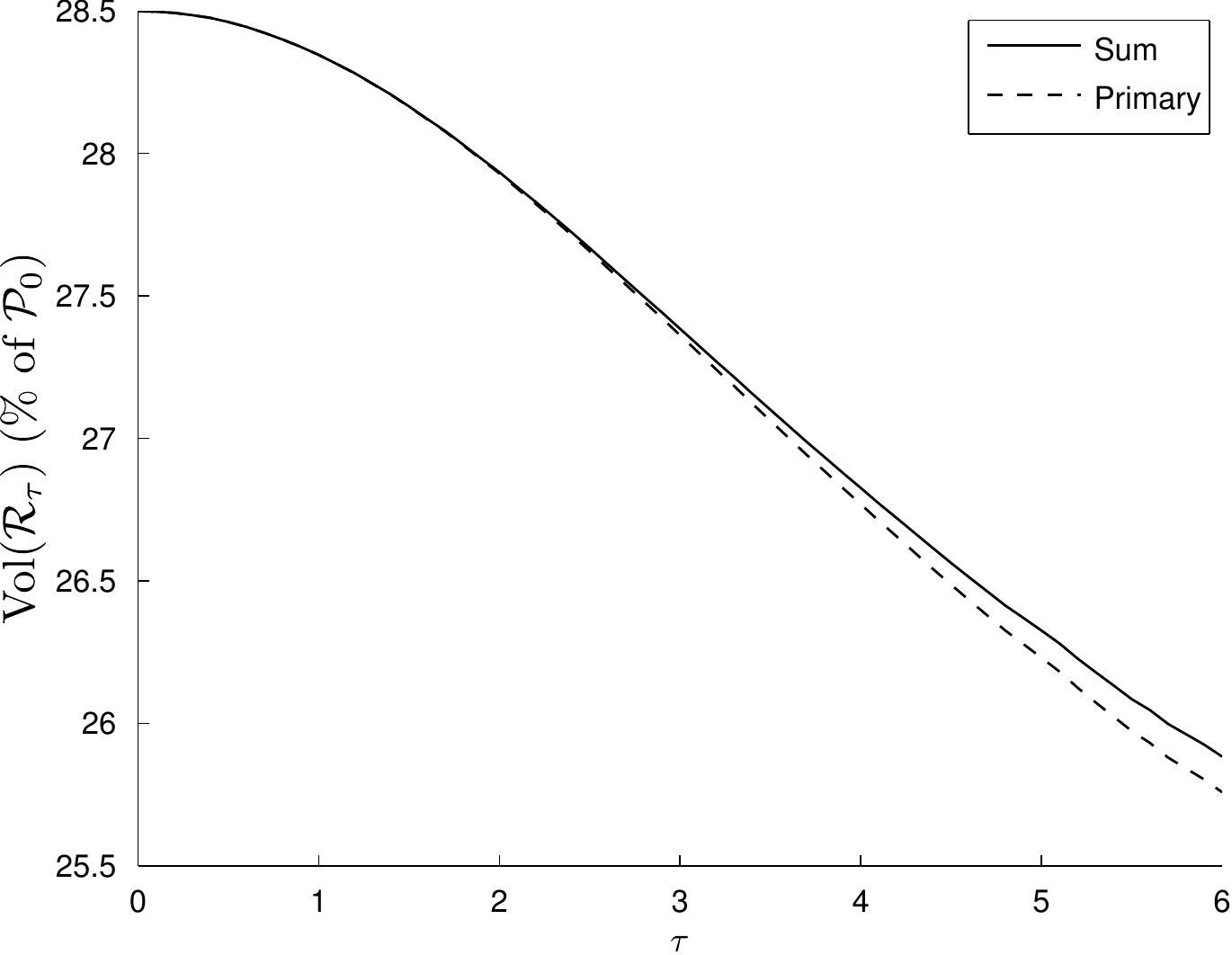}{Dependence of the primary and secondary lobe volumes on transition time $\tau$ for $(A,B,C)=(1,0.5,1.5)$.  The dotted line denotes the primary lobe volume and the solid line denotes the total flux $\Phi$.}{timeDependence}{2.8in}

We performed two checks on the accuracy of the numerical integrations. For $\tau = 0$, the intersection curves $\cI_0 = \cI_\tau$ correspond to the points
\begin{gather*}
     z = \cos^{-1}\big( \tfrac B A( 1- \sin x) - 1 \big), \\
    y = \cos^{-1}\big( \tfrac A C( 1- \sin z) - 1 \big),
\end{gather*}
as $x$ ranges over $[\tfrac{\pi}{2}, \tfrac{5\pi}{2}]$.
Since analytical expressions exist for both the intersection curves and the lobe boundary itself, it is straightforward to numerically compute the lobe volume directly from the 2D integral in \Eq{StokesTheorem}.  These computations revealed that the relative error in the action-flux computations was, on average, on the order of $10^{-5}$.

As a second check, and for nonzero $\tau$, we estimated $\Phi$ using a Monte Carlo simulation.  We uniformly seeded $\cP_0$ with $N$ points, advected each point to $t=\tau$, and determined the number $N_{in}$ of orbits of sample points that lay within the $\cF^k_\tau$ at $t=\tau$.  An estimate for $\Phi$ is then
\[
  \Phi_{MC} = \Vol(\cP_0)  \frac{N_{in}}{N},
\]
using \Eq{analyticalVolume}.  The relative error in any realization of this Monte Carlo computation is estimated by \cite[\S7.6]{Press92}
\beq{MCError} 
  \frac{1}{\sqrt{N}} \sqrt{\frac{N}{N_{in}} - 1}.
\eeq
For $N = 10^6$, the difference between the action-flux computations and the Monte Carlo simulations was typically less than \Eq{MCError}.  For larger values of $\tau$ the difference between the two methods increases; this is likely due to resolution issues with the intersection curves $\cI_\tau$, as neighboring points along these curves begin to separate significantly as $\tau$ becomes large.

Note that a Monte Carlo computation of the flux is only feasible due to the simple nature of $\cP_0$ and $\cF^k_\tau$.  Since the boundaries of these past- and future-invariant sets are known analytically, it is straightforward to both uniformly sample $\cP_0$ and determine which trajectories lie within $\cF^k_\tau$ at $t=\tau$. This could be computationally prohibitive for more complicated past and future-invariant sets. In such cases, using the action-flux formulas to compute $\Phi$ typically requires less Lagrangian information.

\section{Example: Microdroplet Flow} \label{sec:microdroplet}
As a second application of our theory, in this section we study transport between two halves of a droplet moving through a serpentine microfluidic channel mixer. Microfluidic devices have numerous applications, for example to detect specific antigens in the blood \cite{Chin11}, perform
macromolecular or cellular assays and analyze DNA \cite{Beebe02}, and even filter circulating tumor cells from the blood for early stage cancer diagnosis and metastasis detection \cite{Nagrath07, Schattner09}.  They have also been key components of process intensification in industry \cite{Ottino04a, Meijer09}, an effort to decrease processing time, make more efficient use of raw materials, and gain greater functionality 
and sensitivity with respect to device size.  

Many applications of microfluidics require thorough mixing; however, small length scales or high fluid viscosities often force a Stokes flow regime in which mixing by molecular diffusion alone can be impractically slow \cite{Ottino04b}.
Consequently, chaotic advection is required \cite{Aref84, Wiggins04}, and designing
devices in which this occurs is of much interest \cite{Stroock02, Bringer04, Stone05}.  

We model a channel mixer for which the mixing occurs within a droplet formed by the injection of equal volumes of fluids $A$ and $B$, see \Fig{2DChannel}(a). For simplicity we assume that the droplet is a sphere and that each fluid occupies one hemisphere at $t=0$. The plane separating $A$ and $B$ at $t=0$ is the \emph{injection plane}, denoted by $\cU_0$. Subsequent motion of the droplet through the serpentine microchannel, 
as sketched in \Fig{2DChannel}(b), will induce a time-dependent vector field $V(\bx,t)$ within the droplet, with the goal of 
stretching and folding the initial interface to enhance the mixing by molecular diffusion.  Finally, after a transition time $\tau$, the droplet is extracted by dividing it into two hemispheres along an \emph{extraction plane}, $\cS_\tau$.  

\InsertFig{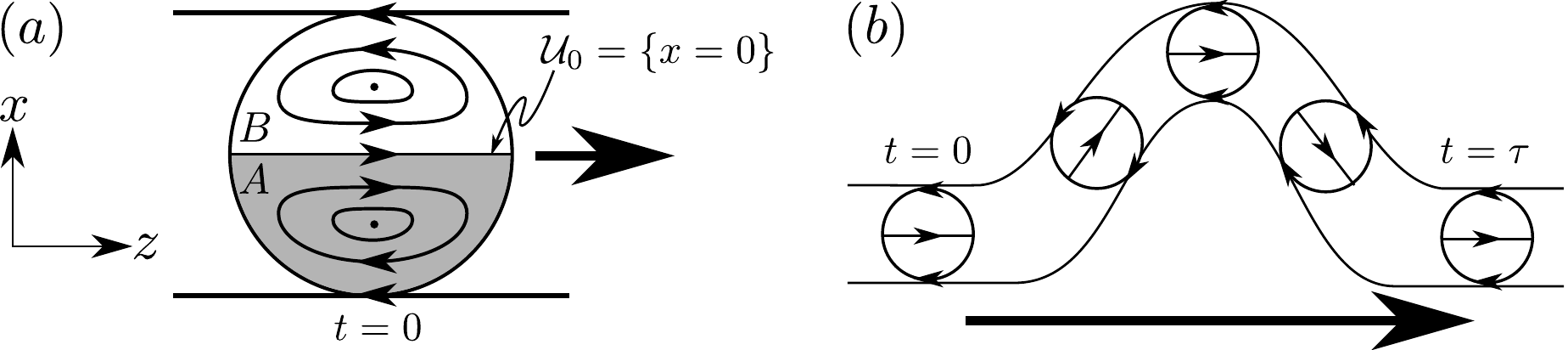}{(a) Streamlines in the $xz$-plane (comoving frame) and initial locations of the fluids $A$ and $B$ and the injection plane $\cU_0$ that separates them.  (b) Streamline dependence on channel shape.}{2DChannel}{.85\linewidth}

We assume $\cU_0$ is invariant under $P(\bx) = V(\bx,0)$ so that the hemispheres $A$ and $B$ are each past invariant.  Similarly, we assume $\cS_\tau$ is invariant under $F(\bx) = V(\bx,\tau)$. Extending the dynamics to $t \in \bR$ in this way gives a flow that is transitory in the sense of \Eq{transitoryODE}.  Our goal is to compute the fraction of $A$ in each of the extracted hemispheres as a function of the microchannel shape.

For simplicity we assume that the droplet remains spherical as it moves through the serpentine channel. We also assume that it is always in contact with the channel walls, so that---in the lab frame---the velocity at these contact points must be zero.  For a straight channel, the resulting steady flow is axisymmetric with a vortical recirculation, such as that sketched in \Fig{contours}(a), and this is confirmed experimentally \cite{Sarrazin06} for a ``bullet-shaped'' droplet in a straight channel with rectangular cross-section. It is reasonable to think that these results would extend to a droplet that has sufficient surface tension to maintain a spherical shape, as well as to a more symmetric, circular channel.  Finally we assume that the center of mass velocity of the droplet remains parallel to the channel walls so that, as the channel bends, the droplet's velocity field rotates to keep its axis of symmetry parallel to the walls, as sketched in Figs.~\ref{fig:2DChannel}(b) for 2D, and \ref{fig:contours}(b) for 3D.

\InsertFig{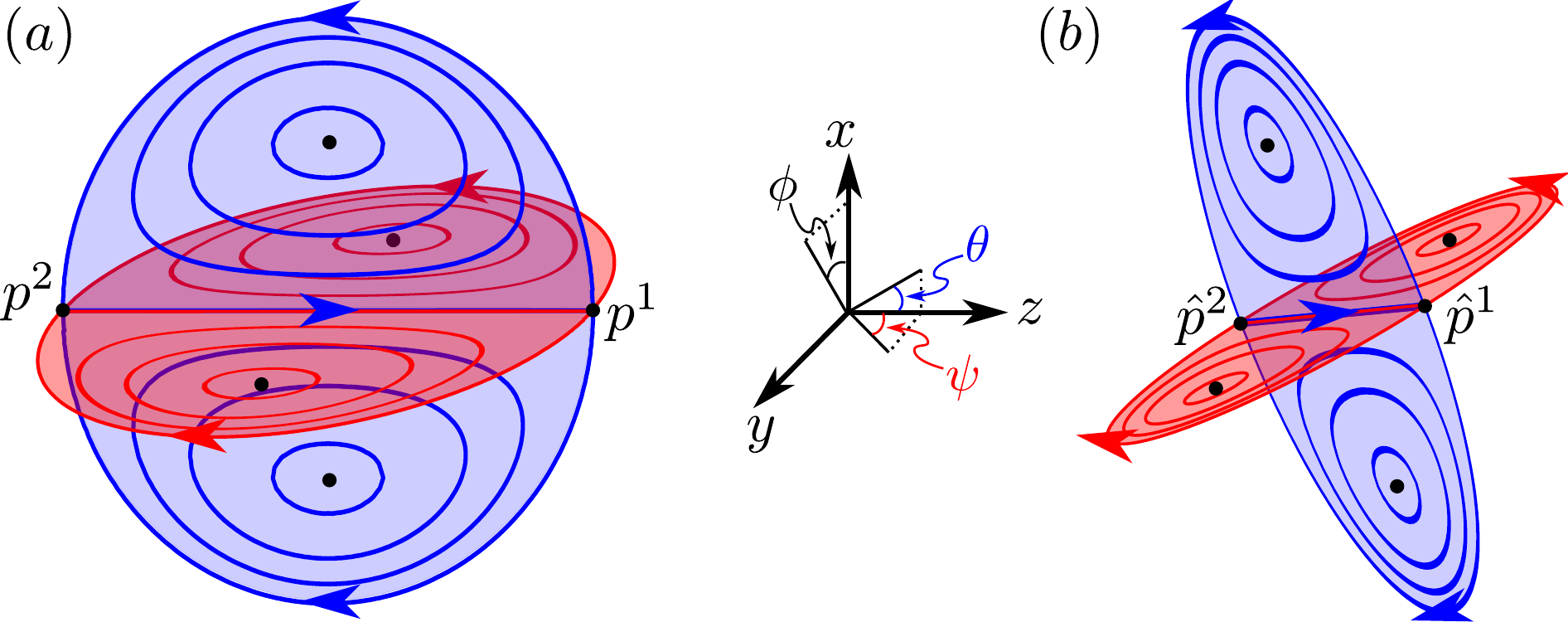}{(a) Streamlines for the droplet vector field $V_0$ in the $xz$- (blue) and $yz$- (red) planes.  (b) Corresponding contours of the rotated vector field with angles  $\theta,\, \psi$ and $\phi$ as defined by \Eq{singleRotations}.}{contours}{.75\linewidth}

\subsection{Stationary Model} \label{sec:stationaryDroplet}
The steady, low-Reynolds-number flow inside a spherical droplet subject to a uniform external creeping flow $\mathbf{U} = -U\hat z$ (the Hadamard-Rybczynski problem) is known analytically \cite{Hadamard11}.  Following \cite{Stone05}, we take this flow to model the steady motion within a droplet moving through a straight section of microchannel.  To nondimensionalize, we normalize lengths by the droplet radius, $D$, velocities by $U$, and time by $4D(1+\mu)/U$, where $\mu$ is the ratio of the viscosity of the fluid within the droplet to that of the exterior fluid. The nondimensional vector field within the droplet in a comoving reference frame is given by
\[
 V_0(\bx) = \left( \begin{matrix}
   2xz \\
   2yz \\
   2(1 - 2x^2 - 2y^2 - z^2)
 \end{matrix} \right).
\]
In this frame, the droplet sits at the origin and its boundary ($r^2 = |\bx|^2 = 1$) is invariant under $V_0$.  Certainly the length scales in typical microfluidic devices imply that the low Reynolds number assumption used to obtain this solution is appropriate. In addition, if the surface tension at the droplet boundary is sufficiently large, the assumption that the droplet remains spherical seems reasonable.

The vector field $V_0$ has two hyperbolic equilibria (stagnation points) $p^1 = (0,0,1)$ and $p^2 = (0,0,-1)$, both on the droplet boundary, recall \Fig{contours}(a).  The corresponding two-dimensional manifolds $\Wu(p^1)$ and $\Ws(p^2)$ form a heteroclinic connection between $p^1$ and $p^2$ that comprises the droplet boundary.  Similarly, the one-dimensional manifolds $\Ws(p^1)$ and $\Wu(p^2)$ form a second heteroclinic connection between these hyperbolic points in the interior along the axis of symmetry.  There is also a ring of elliptic equilibria at $z=0$ with $x^2 + y^2 = 1/2$.

The vector field $V_0$ can be viewed as a sum of two, 2D Hamiltonian vector fields with $(x,z)$ and $(y,z)$ as canonical variables, and Hamiltonians
\beq{2DHam}
 \begin{split}
  H(\bx) &= (1 - x^2 - y^2 - z^2)x, \\
  K(\bx) &= (1 - x^2 - y^2 - z^2)y,
 \end{split}
\eeq
respectively.  Then $V_0$ is equivalent to
\beq{HillsVortexHam}
 V_0(\bx) = \left( \begin{array}{c} \smallskip
                      -\frac{\partial H}{\partial z} \\ \smallskip
                      -\frac{\partial K}{\partial z} \\
                      \frac{\partial H}{\partial x} +\frac{\partial K}{\partial y}
                     \end{array} \right).
\eeq
With this formulation, it is straightforward to show that \Eq{HillsVortexHam} is globally Liouville in the sense of \Def{Liouville}, with 
\beq{dropletLambda0}
 \begin{split}
  \beta_0 &= H\,dy - K\,dx, \\
  \lambda_0 &= (z\dot x + H)\,dy - (z\dot y + K)\,dx.
 \end{split}
\eeq
The vector field $V_0$ contains no swirl: each plane containing the $z$-axis is invariant.  Axisymmetry implies that the dynamics in each such plane is equivalent to that in the $y=0$ plane, for example, which is Hamiltonian with $H(x,0,z)$. Consequently the dynamics of $V_0$ is completely integrable.

\subsection{Transitory System}

In our model, as the droplet moves in the serpentine channel, e.g., \Fig{pipeshape}, transitory time dependence is introduced when the vector field \Eq{HillsVortexHam} rotates to maintain its axis of symmetry parallel to the channel walls. 

\InsertFig{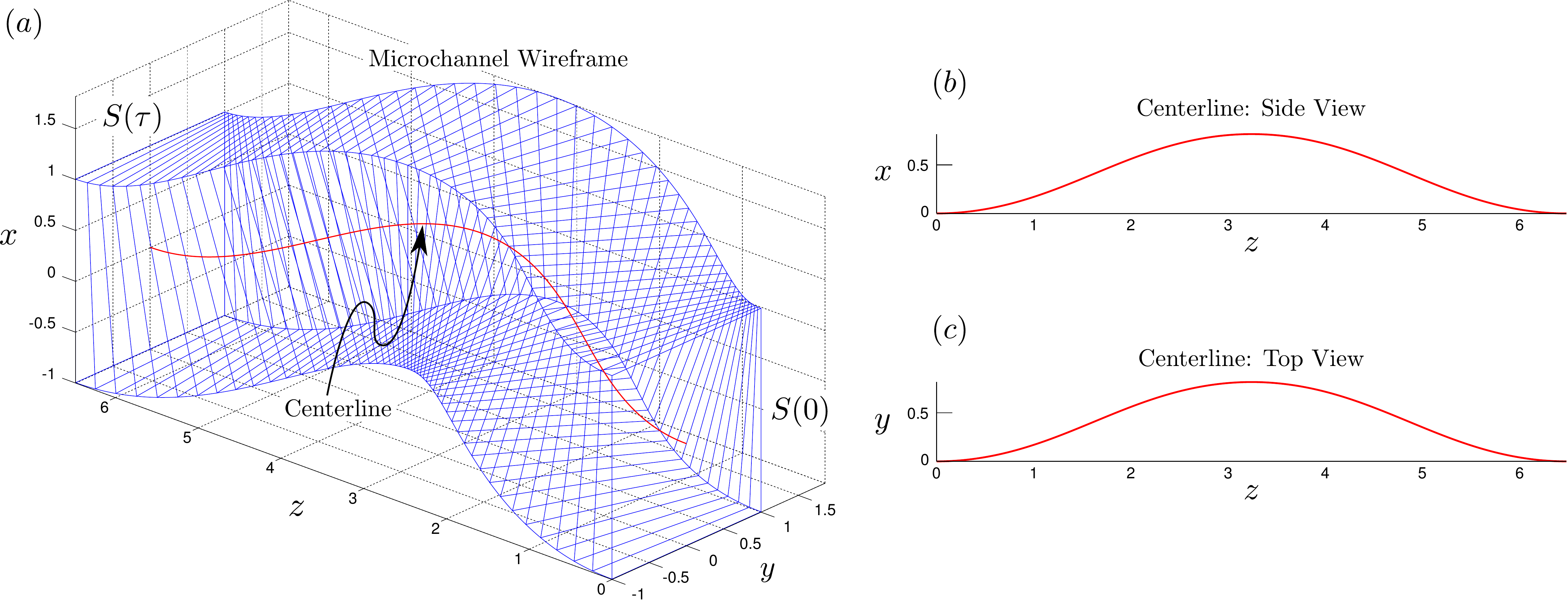}{(a) Wireframe of the microchannel for angles given by \Eq{angles} with $\tau = 3.5$ and $\xi=\pi/8$.  The red centerline is equidistant from the four corners of each cross-section and everywhere tangent to the bulk velocity of the droplet. (b)-(c) Projections of the red centerline onto the $xz$- and $yz$-planes.}{pipeshape}{\linewidth}

We will assume that the fluid $A$ initially occupies one of the hemispheres bounded by the injection plane $\cU_0$. Without loss of generality, we will choose $\cU_0 = \{(0,y,z)\}$,\footnote
{Note that here $\cU_0$ is \emph{not} the unstable manifold of any past-hyperbolic orbit; indeed, its surface area does not vanish in either direction of time.  We nevertheless use a notation consistent with that of the previous sections, i.e. $\cU_0$ is a past \emph{invariant surface} rather than the past \emph{unstable manifold}.  Similar notation is used for the future-invariant surface $\cS_\tau$.}
and suppose that $A$ occupies the ``negative'' hemisphere ($x<0$), denoted by $\cP_0$. Note that $\cP_0$ is past invariant under $V$, since $V_t = V_0$ for $t<0$ by assumption. 
Our goal is to compute the flux $\Phi$ of $A$ from $\cP_0$ to one of the hemispheres bounded by the extraction plane $\cS_\tau$. We will investigate two possible choices,\beq{extractionPlaneChoices}
  \cS_\tau = \{(0,y,z)\} \:\: \textrm{or} \:\: \cS_\tau = \{(x,0,z)\}.
\eeq
Let $\cF_\tau$ denote the ``positive" extraction hemisphere; i.e., $x >0$ or $y>0$, for the cases above, respectively. Note that the hemisphere $\cF_\tau$ is future invariant.
Since volume is preserved, the flux we compute is  equal to that of the fluid $B$ to the complement of $\cF_\tau$. The values of the fluxes of $A$ to the complement of $\cF_\tau$ and $B$ to $\cF_\tau$ follow by volume preservation.

The flux, of course, depends heavily on the choice \Eq{extractionPlaneChoices} of extraction plane.  For example, if $\tau=0$ and  $\cS_\tau = \cU_0$, $\cF_\tau$ will not contain any of fluid $A$: $\Phi= 0$.  However, if the angle between $\cS_\tau$ and $\cU_0$ is $\tfrac\pi2$, as for our second choice, $A$ and $B$ will be equally represented in each extracted hemisphere when $\tau =0$, so $\Phi = \pi/3$, one-quarter of the droplet volume.

To model the shape of the serpentine channel, we will specify the rotations that give its axis and orientation at each time $t$. Let $\theta(t)$ and $\psi(t)$ be the angles between the channel axis and its projection onto the lab-fixed $xz$- and $yz$-planes, see \Fig{contours}. The angle $\phi(t)$ will represent an additional torsion about the axis of symmetry. These ``Tait-Bryan angles"  correspond to the matrices
\beq{singleRotations} \small 
 R_y(\theta) = \left[ \begin{matrix}
                     \cos \theta & 0 & \sin \theta \\
                     0 & 1 & 0 \\
                     -\sin \theta & 0 & \cos \theta
                   \end{matrix} \right],
                   \:
  R_x(\psi) = \left[ \begin{matrix}
                    1 & 0 & 0 \\
                    0 & \cos \psi & \sin \psi \\
                    0 & -\sin \psi & \cos \psi
                   \end{matrix} \right],
		   \:
  R_z(\phi) = \left[ \begin{matrix}
                    \cos \phi & -\sin \phi & 0 \\
                    \sin \phi & \cos \phi & 0 \\
                    0 & 0 & 1
                   \end{matrix} \right],
\eeq
and give an overall rotation
\beq{rotation}
 R(t) = R_y(\theta(t))R_x(\psi(t))R_z(\phi(t)).
\eeq
These are not the typical ``right-handed'' rotations in $\bR^3$, but are designed to give the horizontal, vertical, and torsional angles along the microchannel. In the frame of reference of the center of mass of the  droplet, the time-dependent vector field is then the pushforward, \Eq{pushFowardV}, of $V_0$ by $R(t)$,
\beq{droplet}
 V(\bx,t) = R_*(t)V_0(\bx).
\eeq
Since the channel is initially aligned with the $z$-axis, $\theta(0) = \psi(0) = \phi(0) = 0$ , and since the process only occurs for $t \in[0,\tau]$, we can formally extend the vector field to $t \in \bR$ by setting 
\[
\dot \theta(t) = \dot\psi(t) = \dot\phi(t) = 0 \quad \textrm{for} \quad t \notin [0,\tau],
\]
so that \Eq{droplet} is transitory.  In similar fashion to \Sec{stationaryDroplet}, denote the hyperbolic equilibria at the leading end of the droplet by $p^1$ and $f^1$ and those at the trailing end of the droplet by $p^2$ and $f^2$, under the past and future vector fields, respectively.  Since the droplet boundary at $r=1$ remains invariant under \Eq{droplet}, the phase space $M$ is simply the closed ball of radius 1, centered at the origin.

Since the droplet is assumed to contact the channel walls, no-slip boundary conditions and \Eq{HillsVortexHam} imply that, in the lab frame, the droplet moves with a nondimensional speed of two. Since its velocity is aligned with the channel axis, the position of the droplet center, $\bc(t)$, satisfies the initial value problem
\beq{centerline}
  \dot{\bc}(t) = 2R(t) \hat z \;, \quad \bc(0) = \bf{0} .
\eeq
The function $\bc(t)$ also defines the channel center as a function of the nondimensional time $t$; moreover, since the dimensional arclength along the channel is simply $s = 4D(1+\mu)t$, $\bc(t)$ also defines the channel center as a space curve, recall \Fig{pipeshape}.

Note that from \Eq{rotation}, the torsion $\phi$ drops out of \Eq{centerline}, as it should. Torsion, however, does affect the channel shape. Indeed, assuming that the channel maintains its shape in the direction perpendicular to the current channel axis, $\dot \bc(t)$, then a point $\bw(0)$ on the channel wall at the injection point evolves to $\bw(t) = \bc(t) + R(t) \bw(0)$. 
In order that this model correspond to a physically realizable channel, the walls must correspond to an embedded submanifold---there can be no self-intersections. For a circular cross-section, self-intersections will occur if the local radius of curvature of any bend in the channel is less than the cross-sectional radius.  For rotations \Eq{rotation}, a circular channel with a nondimensional radius of one will be physically realizable when
\beq{curvatureRadius}
  \dot \theta^2(t) \cos^2 \psi(t) + \dot \psi^2(t) < 4 .
\eeq
For example, it is sufficient that both $|\dot\theta|$ and $|\dot\psi| < 2$. 
If the cross-section is not circular, it is more difficult to obtain an analytical requirement.

For the computations below, we will use the shape functions
\beq{angles}
 \theta(t) = \psi(t) =  \xi \sin( {2\pi t/\tau}) \;, \quad \phi(t) = 0,
\eeq
for $t \in [0,\tau]$, recall \Fig{pipeshape}, and fix $\theta(t) = \psi(t) = \phi(t) = 0$ outside the transition interval. Since these angles vanish at $t=0$ and $t=\tau$, both $P$ and $F$ are equal to $V_0$ in \Eq{HillsVortexHam}. 
For this model, there is a critical transition time $\tau^*$ for each amplitude $\xi$, below which the channel walls would self-intersect. For a circular channel this is easily obtained from \Eq{curvatureRadius}:
\[
  	\tau^*_{circ}(\xi) = \sqrt{2}\pi \xi .
\]

When the channel has a square cross-section, as shown in \Fig{pipeshape}, a slightly larger $\tau^*$ is required because the self-intersection first occurs at the corners; a numerical solution for two cases gives 
\[
	\tau^*_{sq} = \left\{\begin{array}{ll}
				 2.4675, &  \xi = \tfrac\pi8 \\
				 4.9348, &  \xi = \tfrac\pi4
					\end{array}\right. .
\]
In terms of dimensional variables, the physical channel length is $L = U\tau$, and so, for a given viscosity ratio, these requirements imply a lower bound on on the ``inverse aspect ratio" of the channel,
\[
	\frac{L}{D}  >  {4(1+\mu)\tau^*} .
\]
Even when $\tau < \tau^*$, the model \Eq{droplet} still corresponds to a droplet that is forced to rotate; however, to realize these rotations in a lab, some other mechanism, such as electromagnetic manipulation of a charged microdroplet \cite{Lyuksyutov04, Shi10} would have to be used.

We will use the model \Eq{angles} to investigate the effects of the transition time  $\tau$ and amplitude $\xi$ on the transported flux. As we noted above, given a fixed diameter and viscosity ratio, the transition time $\tau$ is a proxy for the channel length. The oscillation amplitude $\xi$ controls the magnitude of the bends in the channel.

\subsection{Computation}\label{sec:DropletComputations}
To compute $\Phi$, we employ the action-flux results of \Sec{lobeVolumes}.  For any $\tau > 0$, there exist well-defined lobes containing all the fluid $A$ in $\cF_\tau$; these lobes are bounded at any time by slices of the orbits of the injection and extraction planes, $\cU_t$ and $\cS_t$, and the invariant droplet boundary $\partial M$ (see \Fig{dropletLobes}).
Since $\cU_0$ and $\cS_\tau$ are not hyperbolic manifolds, the intersections of their orbits in any time-$t$ slice are not heteroclinic, as was the case in \Sec{abc}; however these intersection curves are still needed for the computation of the lobe volumes according to the action-flux formulas.  They are \emph{interior} to the droplet (except at possibly two points on $\partial M$) and we denote them at time $\tau$ by
\[
  \cI^\textrm{int}_\tau  = T(\cU_0) \cap \cS_\tau.
\]
where $T$ is again the transition map \Eq{transitionMap}.
The intersections of $T(\cU_0)$ and $\cS_\tau$ with $\partial M$ itself are also important.  They lie on the invariant spherical \emph{boundary} and are denoted by
\[
  \cI^\partial_\tau = T(\partial\cU_0) \cup \partial\cS_\tau.
\]

\InsertFig{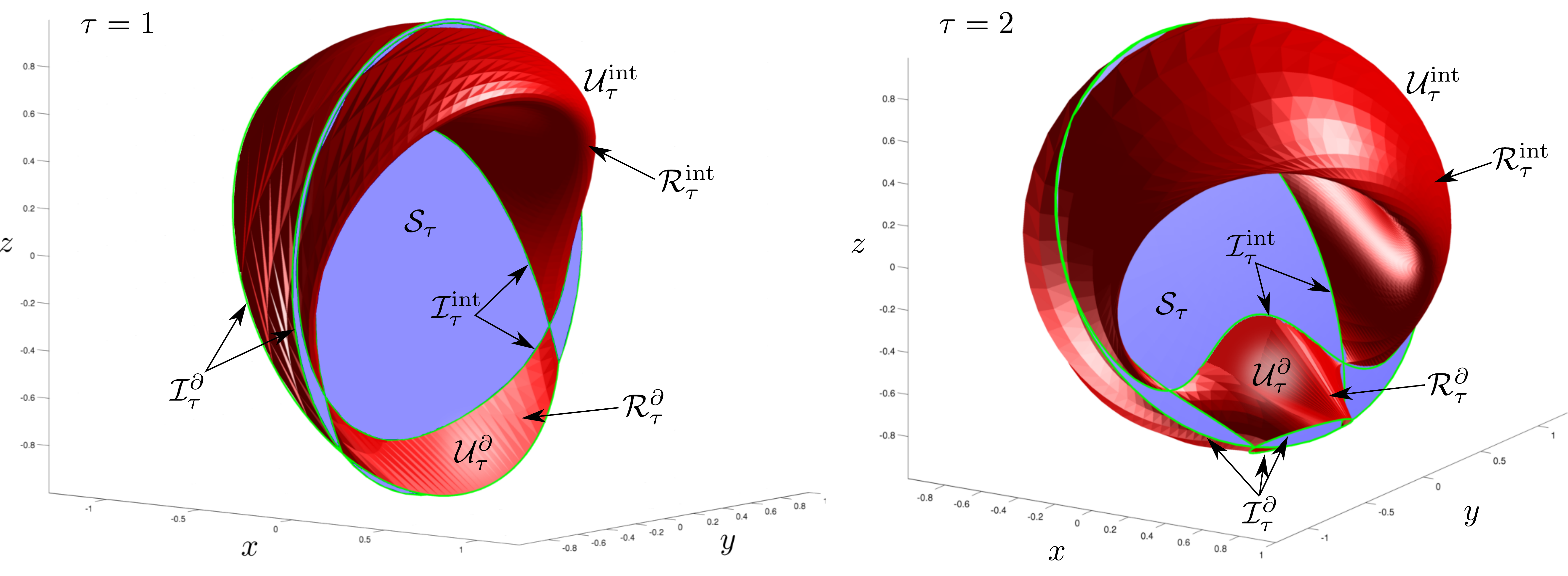}{Lobes contributing to the flux $\Phi$ at $t=\tau$ for $\cS_\tau = \{x=0\}$ and $\xi = \pi/4$. The intersection curves $\cI$ are shown in green and the spherical boundary $\partial M$ is not shown.}{dropletLobes}{\linewidth}

We use a method similar to that in \Sec{abc} to identify the intersection curves.  The injection plane $\cU_0$ has a natural parameterization $G: \cD_1({\bf 0}) \to M$, given by
\[
 G(u,v) = (0,u,v).
\]
Letting $\cW:M \to \bR$ so that the extraction plane $\cS_\tau$ is its zero level set, the interior intersection curves in parameter space,
\beq{rootFindingProblem}
 \cI^\textrm{int}_G = \{ (u,v) \in \cD_1({\bf 0}) \:\: \big| \:\: \cW( T( G(u,v) ) ) = 0 \},
\eeq
are computed using a 2D root finder. Figure \ref{fig:intersections} shows these curves for $\xi=\pi/4$, various values of $\tau$, and both choices for the extraction plane \Eq{extractionPlaneChoices}.
The corresponding curves in phase space at $t=\tau$ are
\[
 \cI^\textrm{int}_\tau = T( G( \cI^\textrm{int}_G ) ).
\]

\begin{figure}[ht]
       \centerline{
         \includegraphics[width=.19\linewidth]{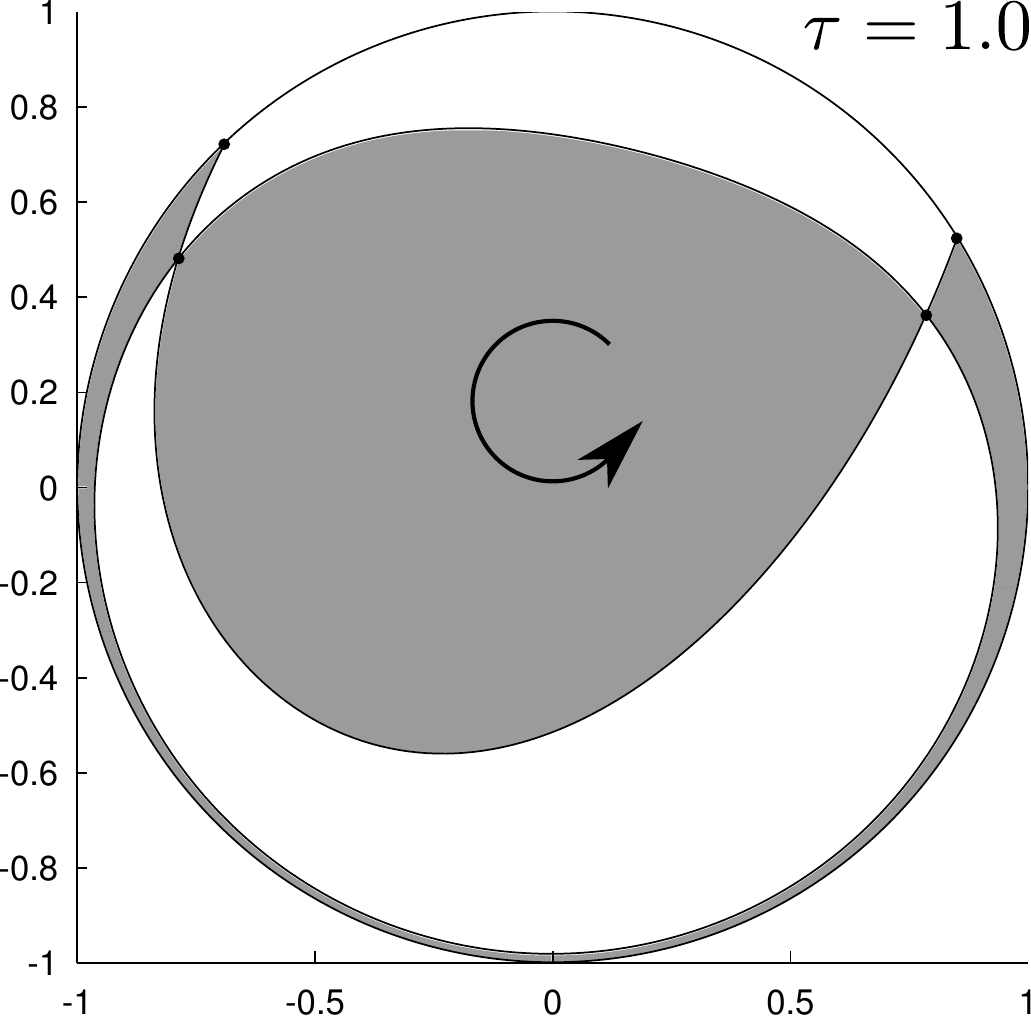}
         \includegraphics[width=.19\linewidth]{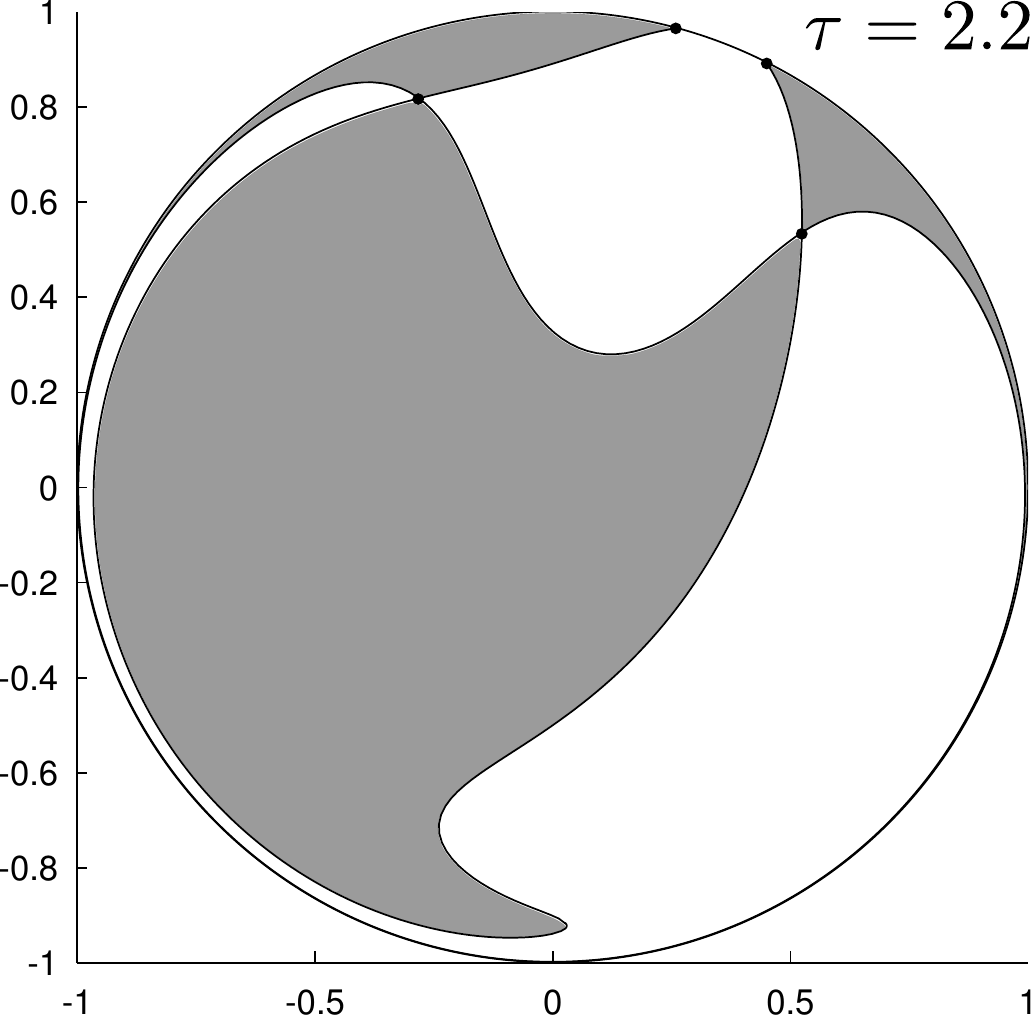}
         \includegraphics[width=.19\linewidth]{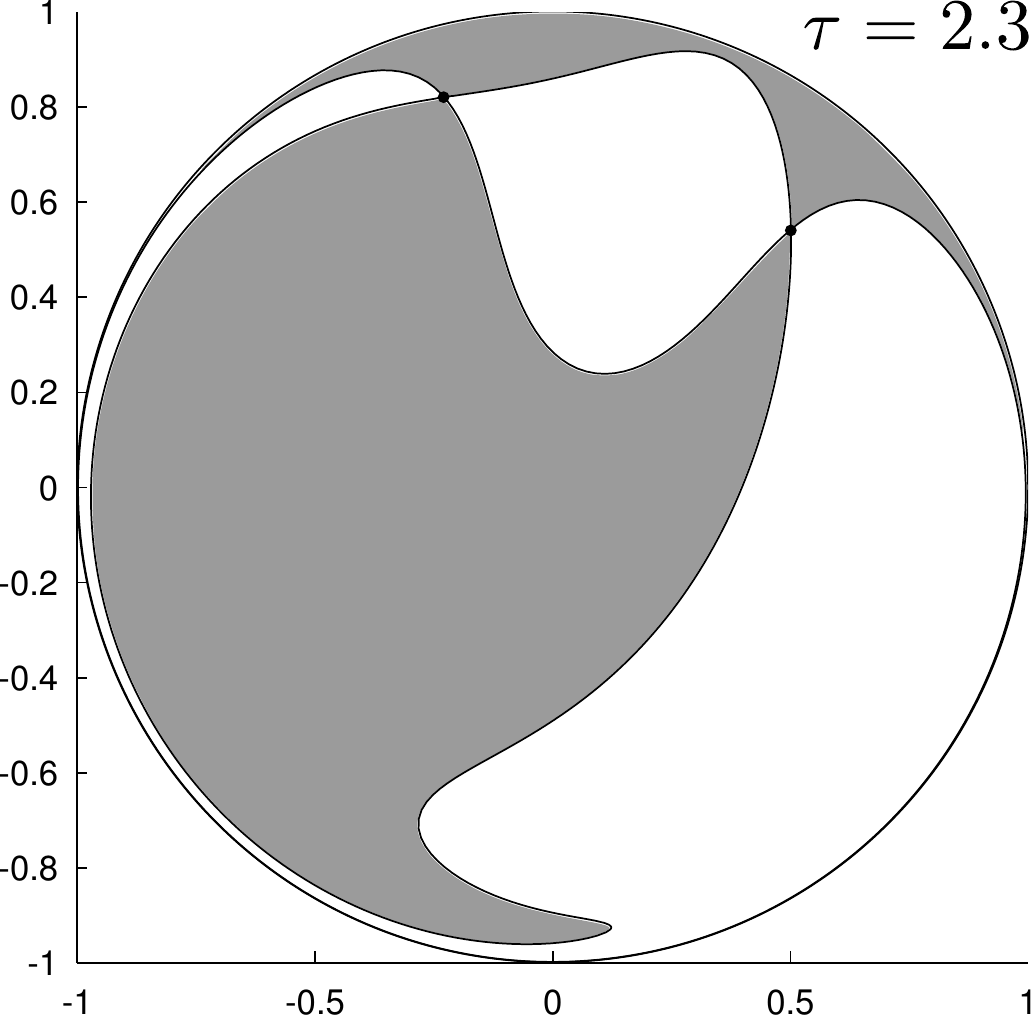}
         \includegraphics[width=.19\linewidth]{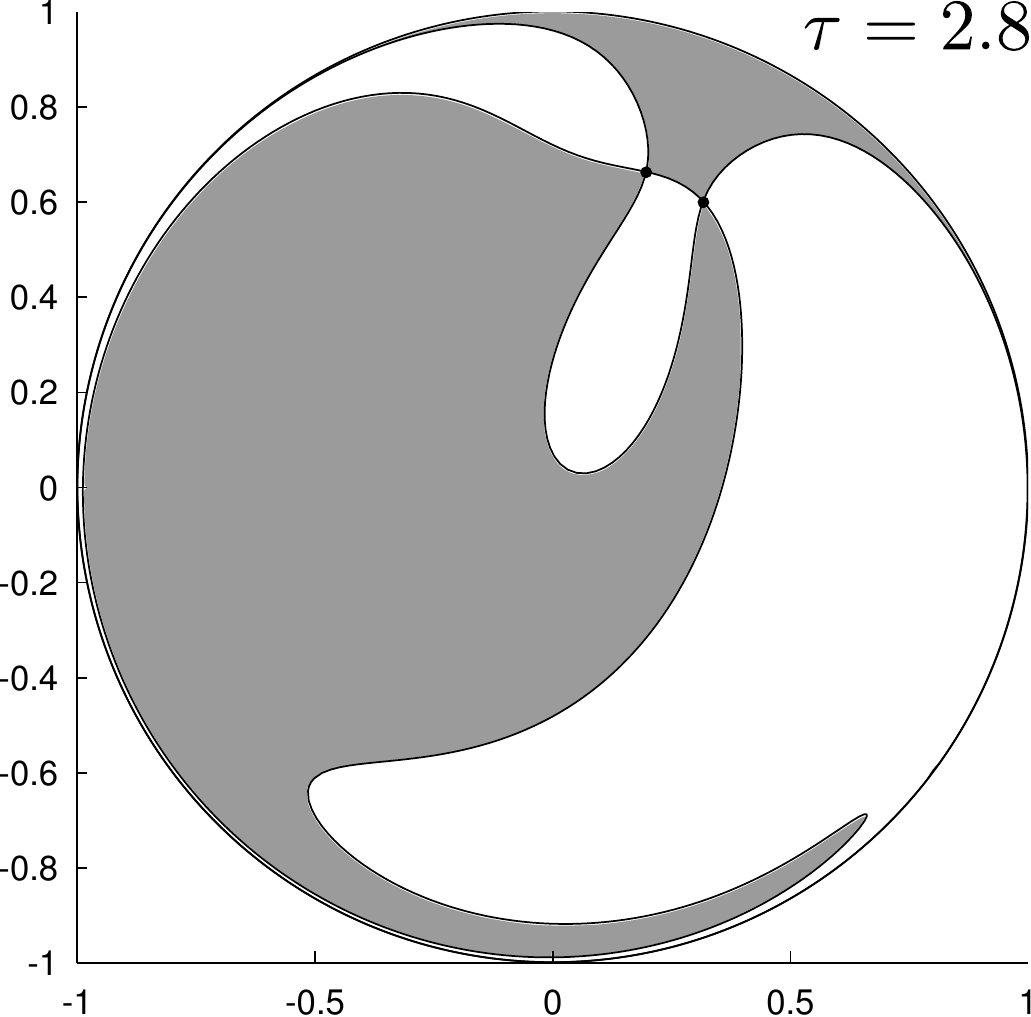}
         \includegraphics[width=.19\linewidth]{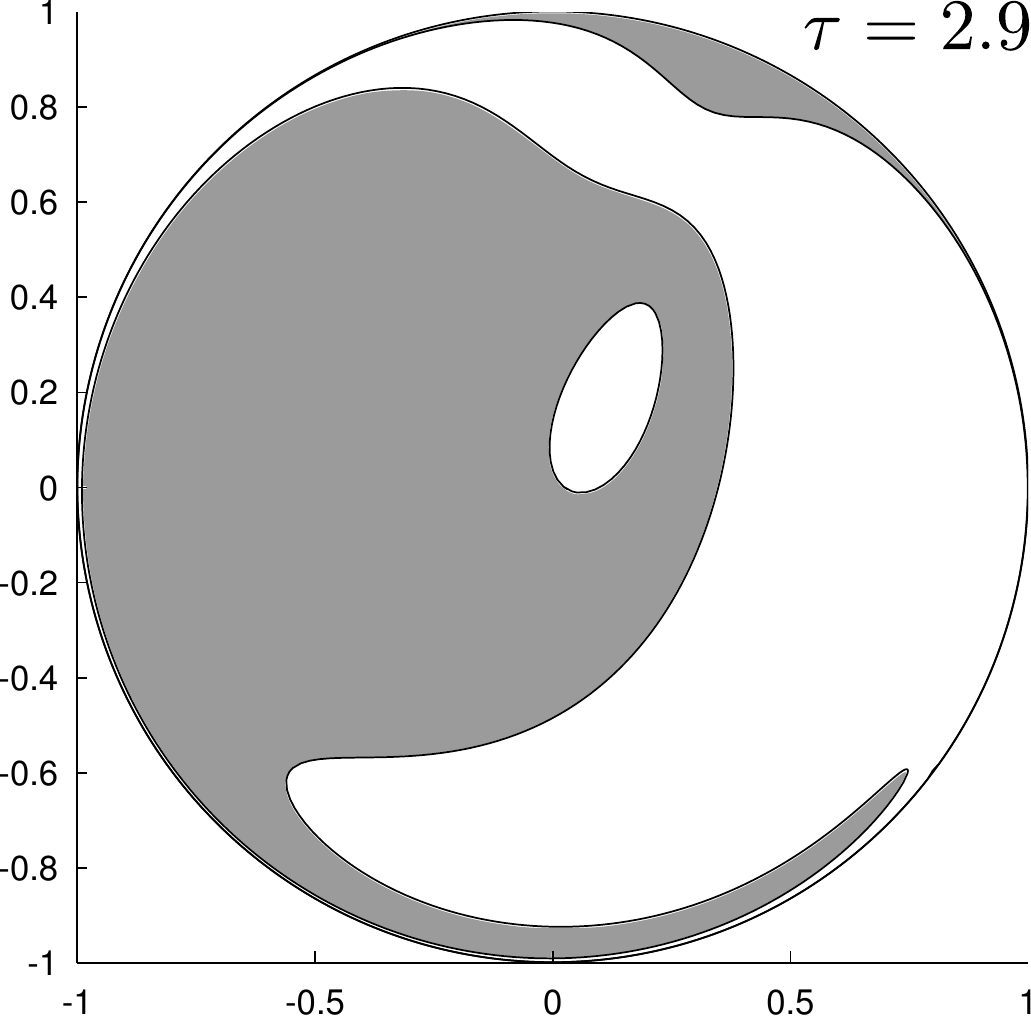}
       }
       \medskip
       \centerline{
	 \includegraphics[width=.19\linewidth]{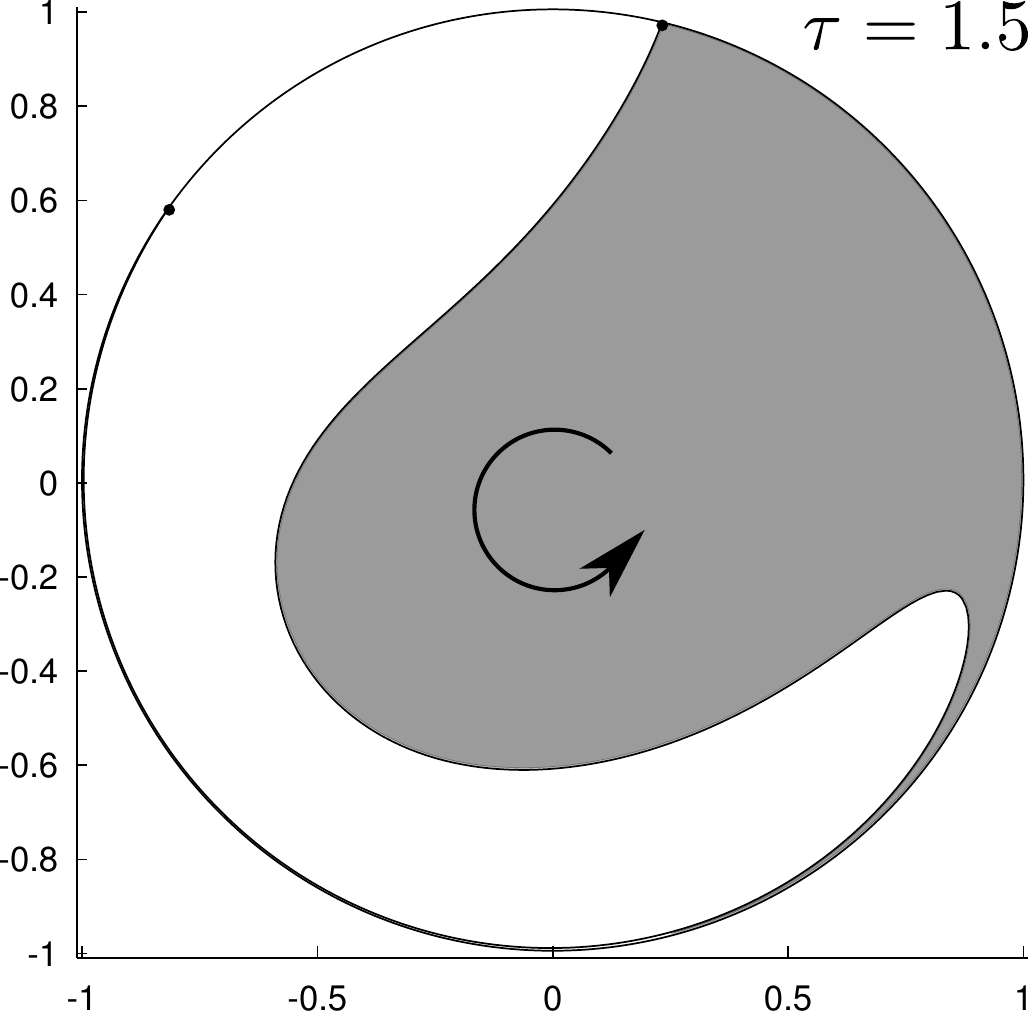}
         \includegraphics[width=.19\linewidth]{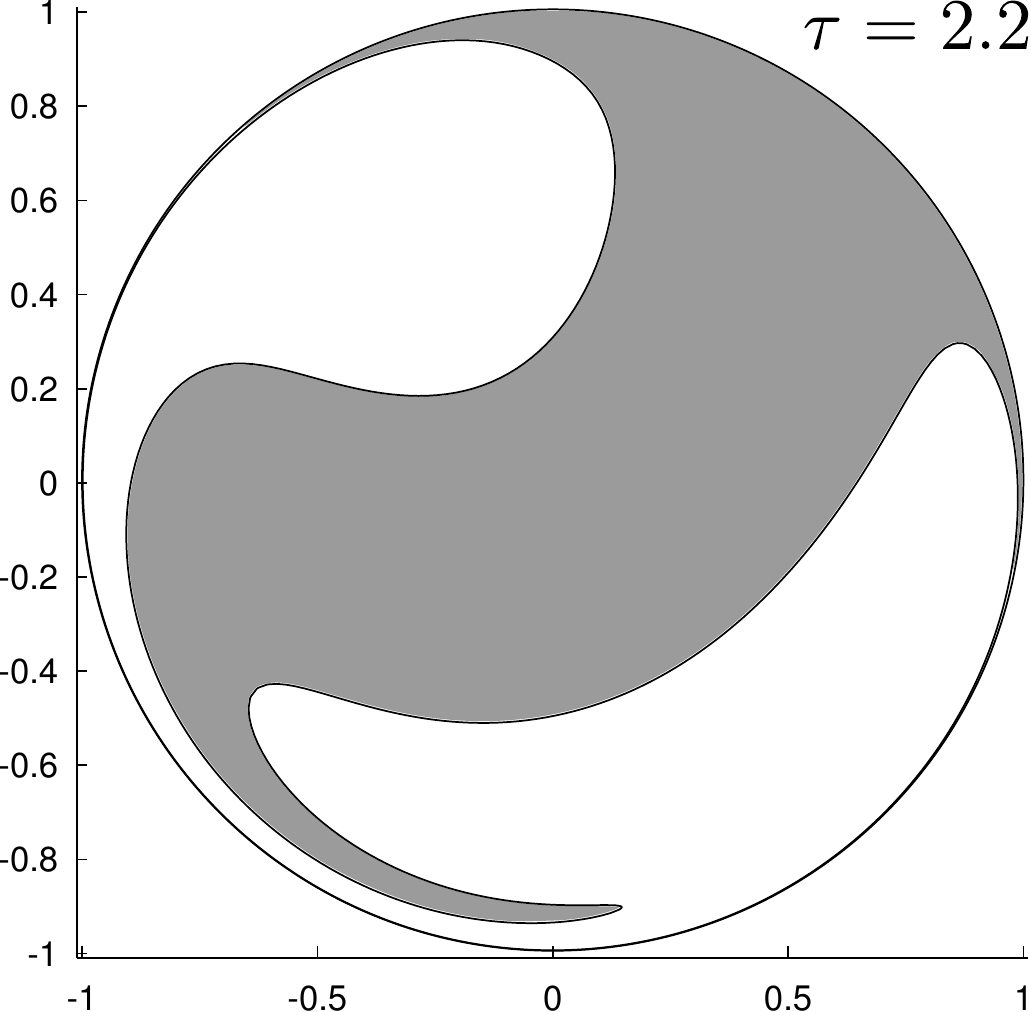}
         \includegraphics[width=.19\linewidth]{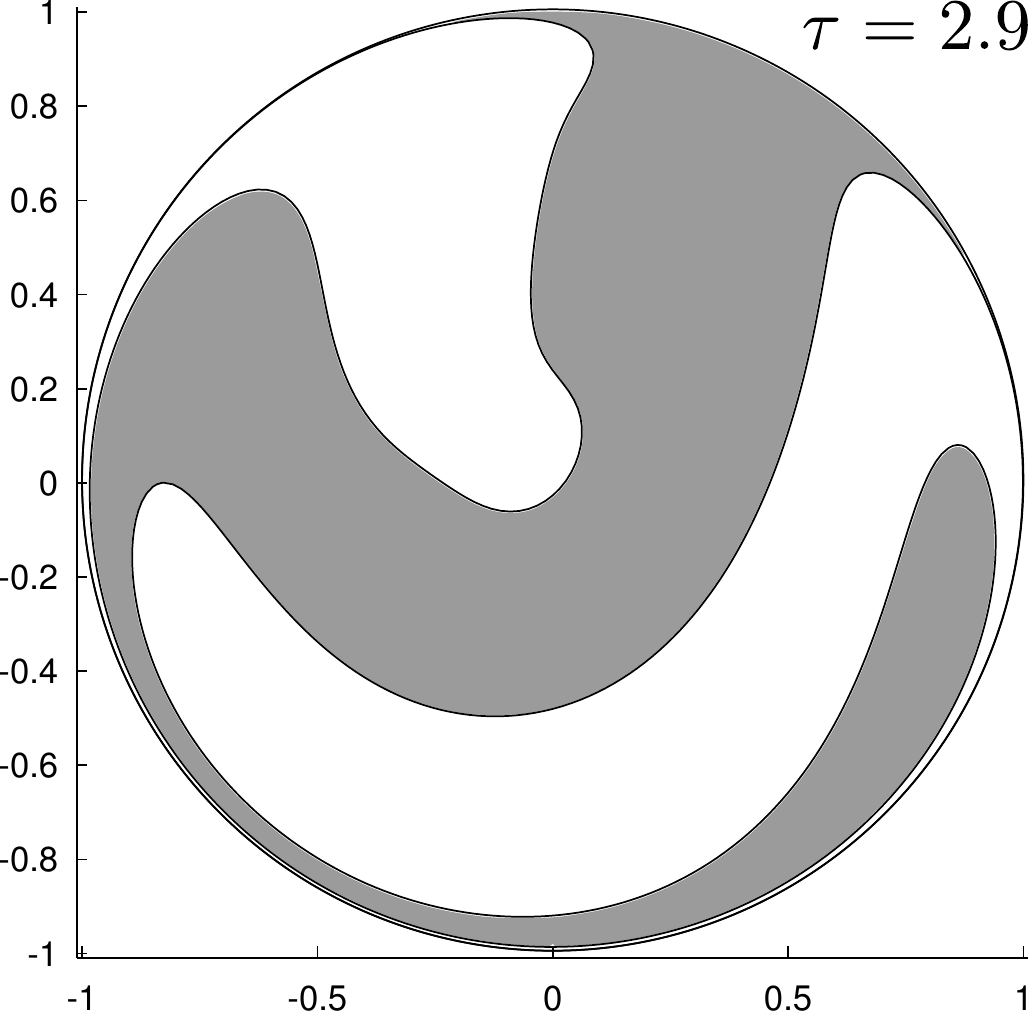}
         \includegraphics[width=.19\linewidth]{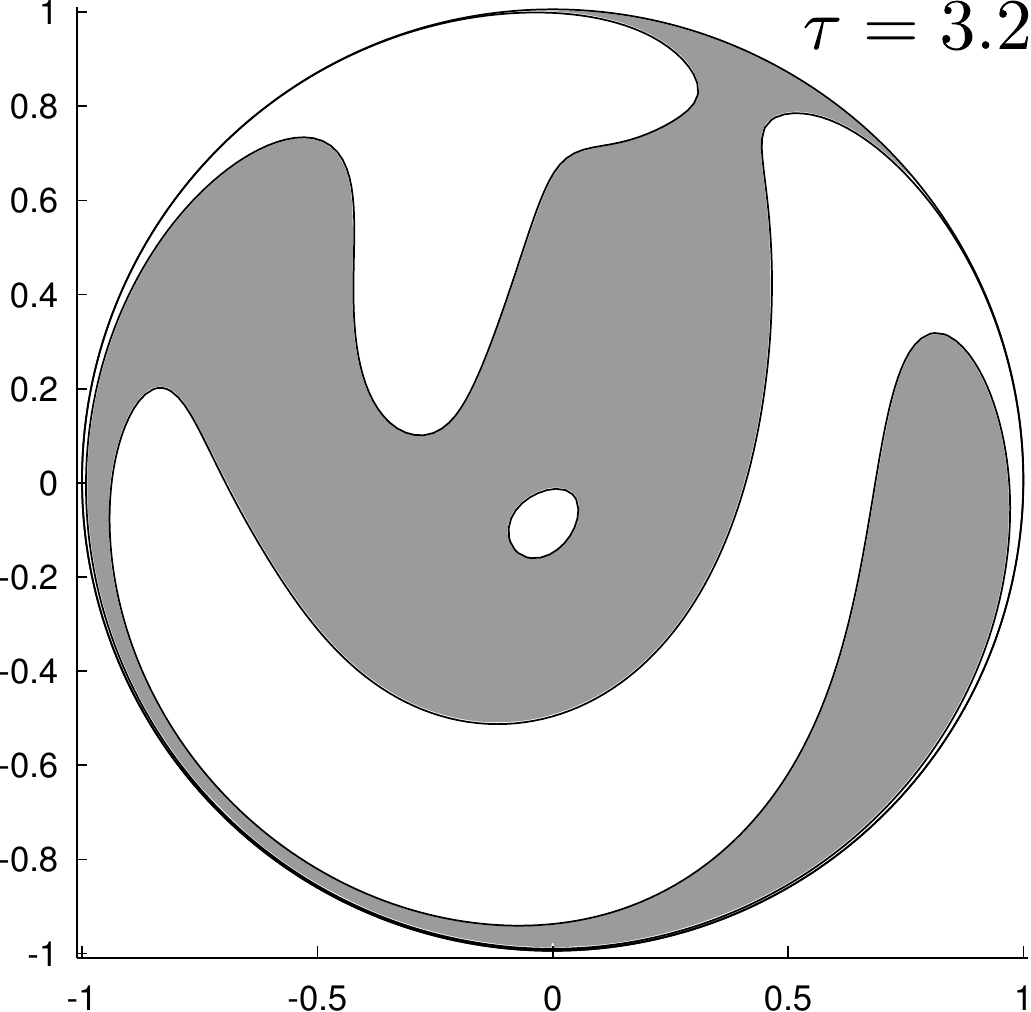}
         \includegraphics[width=.19\linewidth]{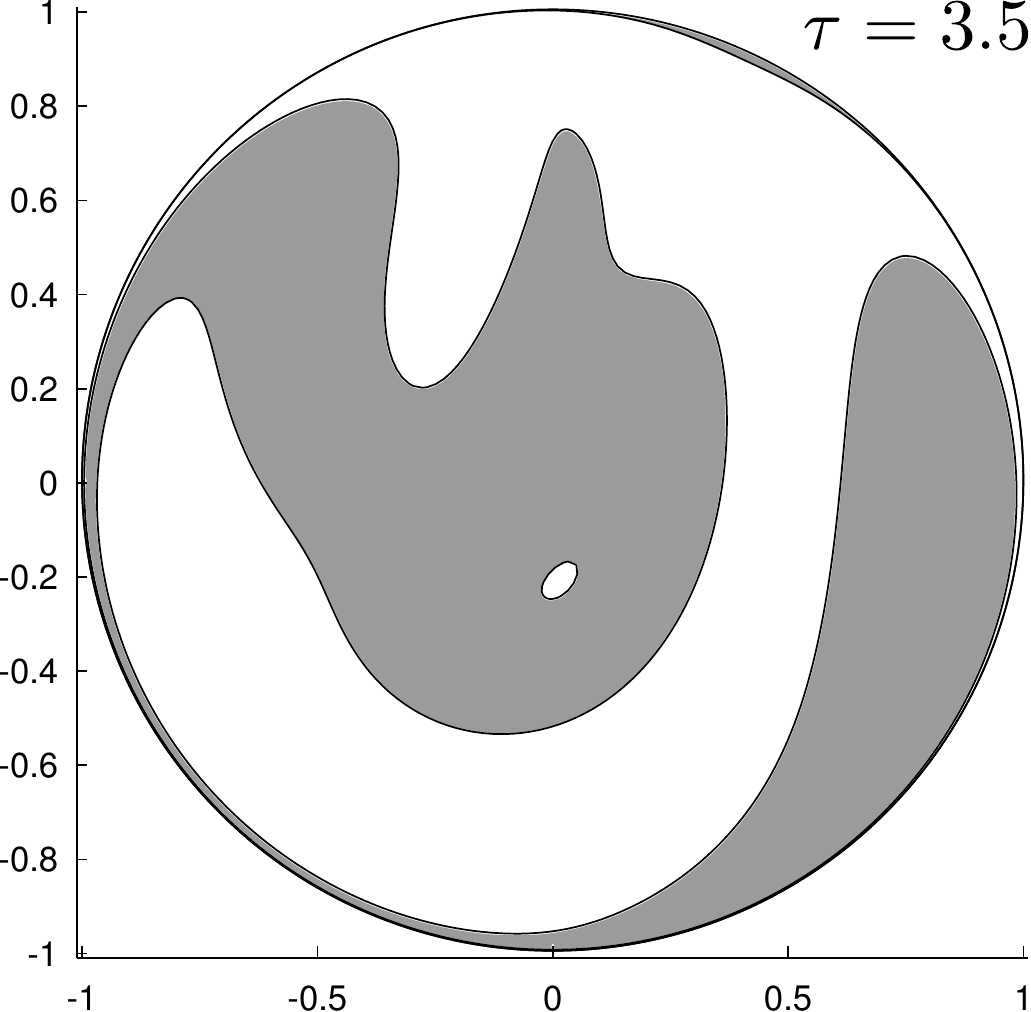}
       }
       \caption{{\footnotesize  Intersection curves $\cI^\textrm{int}_G$ \Eq{rootFindingProblem} with $\xi=\pi/4$ for $\cS_\tau = \{x=0\}$ (top row), and $\cS_\tau = \{y=0\}$ (bottom row).  Regions that map to the positive side of the extraction plane (points for which $\sgn(\cW(T(G(u,v)))) > 0$) are shaded; these correspond to the subsurfaces of $\cU_\tau \subset M$ in \Eq{UIntegral} along whose boundary curves we must integrate to compute $\Phi$.  Representative orientations are shown in the left column.  The accompanying movies \texttt{Microdroplet Intersection Curves:~Extraction Plane \{x=0\}} and \texttt{Microdroplet Intersection Curves:~Extraction Plane \{y=0\}}, corresponding to the top and bottom rows, respectively, show the continuous dependence of $\cI^\textrm{int}_G$ on $\tau$.}
       \label{fig:intersections}}
\end{figure}

To compute the transition map, we use a combination of Cartesian and spherical polar representations of \Eq{droplet} and integrate the vector field using MATLAB's implementation of the Runge-Kutta (5,4) Dormand-Prince pair. The point is that integration in Cartesian coordinates does not respect the invariance of the sphere $\partial M$, while standard spherical coordinates induce singularities at the origin and along the positive $z$-axis. To avoid the latter, we introduce a second spherical representation in which the inclination is measured from the positive $y$-axis.  We monitor both inclination and radius, switching between spherical representations when the inclination drops below some prescribed minimum, and switching to the Cartesian representation near the origin.

We turn now to the application of the action-flux formulas to compute $\Phi$.  It is not hard to show that the rotations in \Eq{singleRotations} are each exact volume-preserving, with generators, \Eq{exactVP},
\[ \begin{split}
 \eta_y &= \frac{1}{2}\cos \theta \sin \theta (z^2 - x^2)\,dy - y \sin^2 \theta (z\,dx + x\,dz), \\ \eta_x &= \frac{1}{2}\cos \psi \sin \psi (y^2 - z^2)\,dx + x \sin^2 \psi (z\,dy + y\,dz), \\
 \eta_z &\equiv 0,
\end{split} \]
respectively.  Consequently, $R(t)$, \Eq{rotation}, is also exact volume-preserving and its generating form, by \Eq{composition}, is
\beq{eta}
 \eta(t) = \eta_y(t) + R_{y*}(t) \eta_x(t).
\eeq
Finally, since $V_0$ is globally Liouville, \Lem{newLambda} implies that the transitory vector field $V(t)$, \Eq{droplet}, is as well, with the  Lagrangian form
\[
 \lambda(t) = R_*(t)\lambda_0 + \cL_{V(t)}\eta(t),
\]
as obtained from \Eq{newLambda} using \Eq{dropletLambda0},  \Eq{rotation}, \Eq{droplet}, and \Eq{eta}.
 
To apply the action-flux formulas, we must identify the appropriate boundaries over which to integrate.  There are two possible types of lobes (interior lobes $\cR^\textrm{int}_t$ and boundary lobes $\cR^\partial_t$, see \Fig{dropletLobes}) and three possible types of bounding surfaces, corresponding to portions of $\cU$, $\cS$ and $\partial M$.
We need only specialize \Eq{mainResult} to integrate the 2-form $\alpha$ over these bounding surfaces.

Since the integral of the two-form $\alpha$ over $\cU_0$ is identically zero, \Eq{mainResult} implies that an integral over $\cU_\tau$ can be reduced to
\beq{UIntegral}
  \int_{\cU^j_\tau} \alpha = \int_0^\tau \left( \int_{\partial \cU^j_s} \lambda \right) ds.
\eeq
A similar simplification applies to $\cS_\tau$, for the two cases that we study.

The integrals over portions of $\partial M$ can be simplified by noting that dynamics on this surface under the stationary vector fields $P=F$ is trivial due to axisymmetry and the invariance of $\partial M$. Indeed, if $\rho$ and $\zeta$ are the azimuthal and inclination angles, then the flow of $F$ reduces to 
\beq{analyticalBdrySoln}
  \vphi^F_t(\rho,\zeta) = \left(\rho, 2 \arctan\left(\tan(\zeta/2) e^{2(t-t_0)}\right)\right)
\eeq
on the droplet boundary.
We use this result to compute the action integrals outside the transition interval. 

Since $\partial M$ is heteroclinic from $p_1$ to $f_2$ under the transitory flow, an integral over a surface $\cB_\tau \subset \partial M$ can be computed using the flow of $F$ in either \Eq{pastAction} or \Eq{futureAction}:
\beq{BIntegral}
  \int_{\cB_\tau} \alpha = \int_{-\infty}^\tau \left( \int_{\vphi^F_{s,\tau}(\partial \cB_\tau)} \lambda \right) ds = 
                             -\int_\tau^\infty \left( \int_{\vphi^F_{s,\tau}(\partial \cB_\tau)} \lambda \right) ds.
\eeq
Here we simply chose to evaluate the integral that converges faster.   It may happen that one the equilibria $f^i$ of $F$ lies on the boundary $\partial \cB_\tau$; in this case, only one of the two integrals \Eq{BIntegral} converges. In the rare case that both equilibria lie on $\partial \cB_\tau$, we can modify the flow $F$ by applying a rotation so as to effectively move the equilibria. Convergence of these integrals can also be accelerated by estimating their exponential tails as discussed in \App{SpeedUp}.

As always, care must be taken to ensure that the intersection curves are oriented to be
consistent with a right-handed outward normal to the lobes. For example, the correct orientation for integration over $\partial \cU$ is indicated in \Fig{intersections}.

\subsection{Results} \label{sec:microdropletResults}
Figure \ref{fig:dropletFlux} summarizes the fluxes computed using the shape model \Eq{angles} for two values of the amplitude $\xi$, the two choices of extraction plane in \Eq{extractionPlaneChoices}, and various transition times.  The curves show the fraction of fluid $A$ found in the positive extracted hemisphere, $\cF_\tau$.
Notice that, in each case, $\Phi$ reaches a maximum at an intermediate transition time.  These maximal fluxes are given in Table \ref{tab:maxFlux}.  

\InsertFigTwo{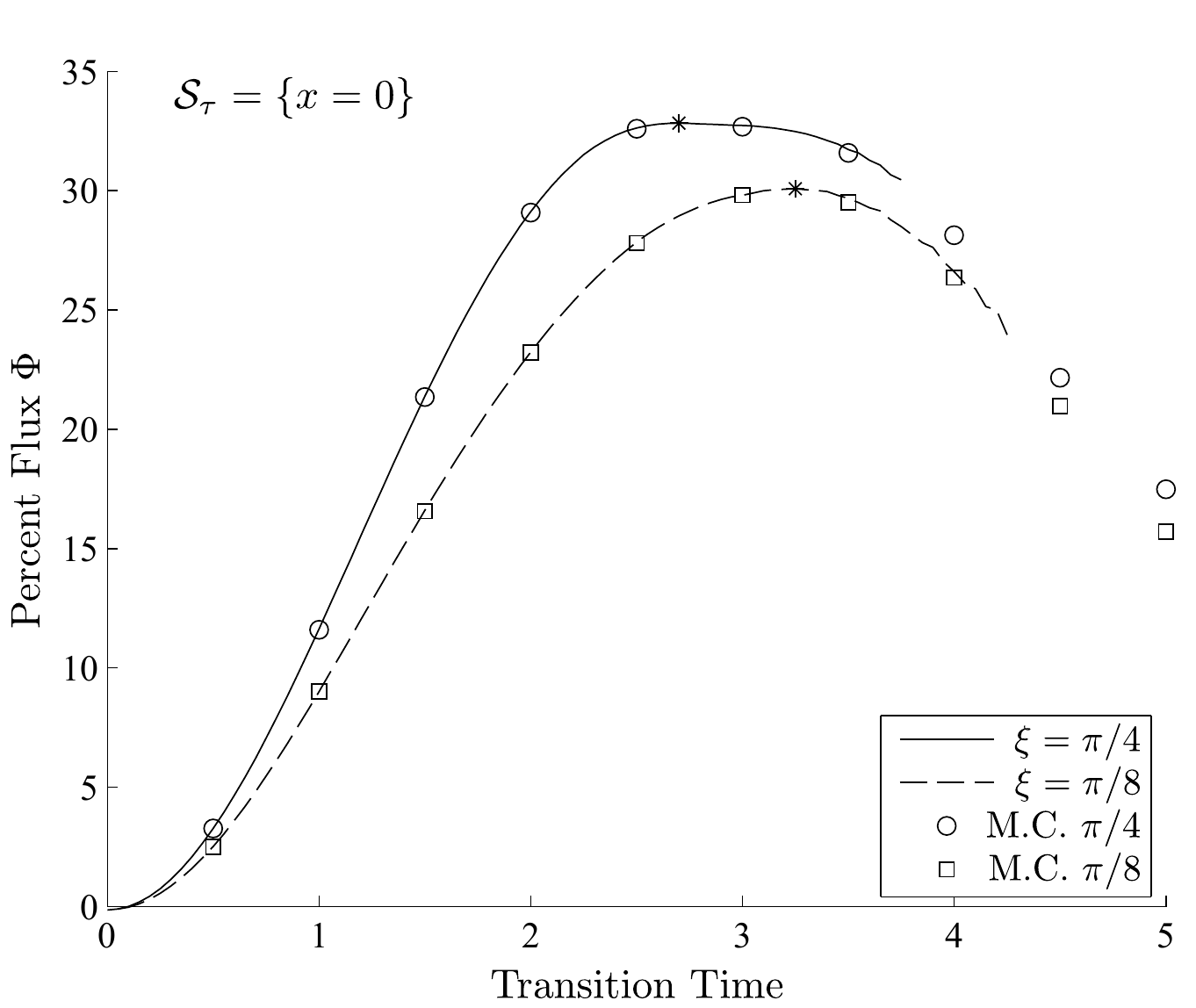}{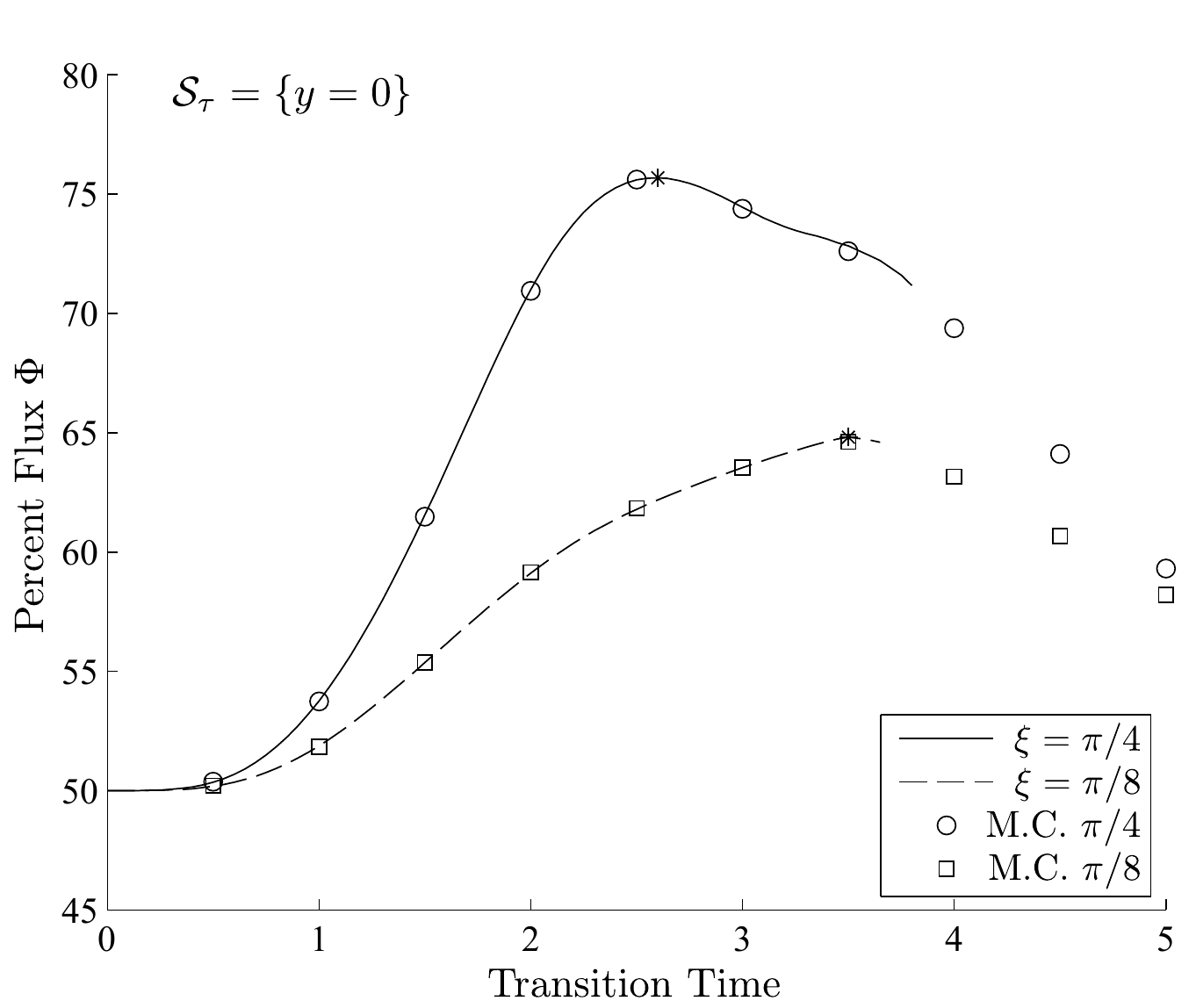}
{Percent composition of fluid $A$ in the positive extracted hemisphere.  Monte Carlo simulation results with $N=10^{6}$ are shown as the circle and square markers (see legend). The maximum computed flux is denoted with an $*$.}{dropletFlux}{.45\linewidth}

\begin{table}[hc]
  \footnotesize
  \renewcommand{\arraystretch}{1.3}
  \centering
  \caption{\footnotesize \label{tab:maxFlux} Maximum percent flux $\Phi$ and corresponding $\tau$.}
  \begin{tabular}{|c|c|c|c|c|}
    \hline
    \multirow{2}{*}{$S_\tau$} & \multicolumn{2}{|c|}{$\xi=\pi/8$} & \multicolumn{2}{|c|}{$\xi=\pi/4$} \\
    \cline{2-5}          & $\tau$              & $\Phi$ (\%)   & $\tau$    & $\Phi$ (\%)   \\ \hline 
    $\{x=0\}$ & 3.25      & 30.22   & 2.7     & 32.97 \\ \hline 
    $\{y=0\}$ & 3.5       & 64.81   & 2.6     & 75.68    \\ \hline 
  \end{tabular}
\end{table}

As $\tau$ increases, the intersection curves $\cI_G$ become increasingly clustered near the boundary of the parameter space, as can be seen in \Fig{intersections}.  Eventually, this clustering becomes so pronounced that the curves cannot be distinguished numerically. This seems to imply that these regions near the boundary contribute little to the flux; however, the transition map expands the small initial differences between these curves by several orders of magnitude. The result is a significant contribution to the overall flux from the unresolved portions of the intersections.  Thus, the computation of $\Phi$ for larger $\tau$ is not numerically feasible using the action-flux formulas. However, the use of both forward and backward integration over the transition interval can ameliorate resolution problems, allowing somewhat larger $\tau$ values to be reached, see \App{SpeedUp}.

To validate our results, we again employ the Monte Carlo technique outlined in \Sec{abcResults}.  Its implementation is simple since we have analytical formulas for the boundaries of the injection and extraction hemispheres $\cP_0$ and $\cF_\tau$.  Since $\cP_0$ is simply half a ball of radius $1$, the overall flux is estimated by
\[
  \Phi \approx \frac23 \pi \frac{N_{in}}{N}.
\]
The results for $N = 10^6$ are indicated in \Fig{dropletFlux} by the open circles and squares. We again find that the difference between the Monte Carlo and action-flux computations of $\Phi$ is typically less than the estimated Monte Carlo error \Eq{MCError}.  Only for large $\tau$ does the difference between the two become significant. Note also that the Monte Carlo computations are inefficient when the flux is small; most of the computational effort is wasted in this case since most sample trajectories do not contribute to $N_{in}$. As a reflection of this, the error also grows, as indicated in \Eq{MCError}. By contrast, the action-flux formulas remain accurate when the flux is small. For example, Monte Carlo simulations do not give accurate results for $\tau < 0.5$ when $S_\tau = \{x=0\}$ in \Fig{dropletFlux}.

An advantage of the Monte Carlo method is that it appears to give reasonably accurate results for larger transition times than the action-flux computations. 
However, this method has several limitations. If the past and future invariant regions had more complex boundaries than the hemispheres in our model, then initialization of the trajectories in $\cP_0$, and determination of whether they are advected to $\cF_\tau$ would require  a high-resolution representation of the lobe boundary surfaces. This is computationally expensive due to the exponential stretching that occurs over the transition interval.  The same problem occurs, even when the regions have analytically simple definitions, if there is more than one lobe and the computation of individual lobe volumes is of interest.
On the other hand, calculation of individual lobe volumes using the action-flux formulas involves no additional effort: each is computed by the action-flux formulas for the particular intersection curves on its boundary.



\section{Conclusions} \label{sec:conclusions}
We have given the first quantitative analysis of transport between Lagrangian coherent structures in aperiodic, three-dimensional flows, and have validated our results using Monte Carlo methods. The coherent structures for the transitory case are simply past- and future-invariant regions, and transport corresponds to the flux from the former to the latter. We rely on the action-flux formulas of \Sec{lobeVolumes} to provide a general framework for computing lobe volumes in $n$-dimensional, globally Liouville flows.  An advantage of these formulas is that lobe volume computations require relatively little Lagrangian information: only the orbits of codimension-two intersections of lobe boundary components need to be computed.  Indeed, we found that high-resolution representations of these intersections can be obtained from a root-finding and continuation method, which is seeded using only a very coarse representation of the lobe boundary itself.  An unusual aspect of the droplet model of \Sec{microdroplet} is that that lobes are not bounded by hyperbolic manifolds of past- or future-hyperbolic orbits---the bounding surfaces are simply invariant under the stationary vector fields $P$ and $F$.

Mixing in droplet models similar to that of \Sec{microdroplet} has been treated in \cite{Stone05} and, in the steady case, in \cite{Vainchtein07}.  Even though these studies computed mixing within a droplet, they did not address finite-time transport between coherent structures.  Our model of smooth transitions in a serpentine mixer is perhaps more realistic than the instantaneous transitions used by \cite{Stone05}, and we plan to compare our present results with direct simulations of the Navier-Stokes equations.

It was observed in \Sec{microdroplet} that there is an optimal transition time that maximizes intradroplet transport.  We would like to further investigate the effects of microchannel shape and choice of the injection and extraction planes on transport.  Ultimately, we hope to find optimal channel shapes to aid in the design of efficient microfluidic mixers.

To make this study more relevant to fluid flows, it will also be valuable to go beyond transport and study diffusive mixing.  The inclusion of diffusive and reactive processes within the transitory framework could help in the development of a quantitative comparison between transport and various mixing norms.  Such an improved understanding will benefit applications ranging from the design of more efficient industrial and microfluidic mixing devices to the effective large-scale recovery of contaminants in the ocean and atmosphere.


\appendix

\section{Some Notation}\label{app:appendix}
Here we set out our notation, which follows, e.g.,~\cite{Abraham78}.
 We denote the set of $k$-forms on a manifold $M$ by $\Lambda^k(M)$. If $\alpha \in \Lambda^k(M)$ and $V_1, V_2, \ldots V_k$ are vector fields, then the pushfoward, $R_*$, of $\alpha$ by $R$ is
\beq{pushForwardF}
    (R_*\alpha)_\bx(V_1,V_2,...,V_k) =\alpha_{R^{-1}(\bx)}((DR(\bx))^{-1}V_1(\bx),\ldots, (DR(\bx))^{-1}V_k(\bx))  .
\eeq
The pushforward can be applied to a vector field $V$ as well:
\beq{pushFowardV}
    (R_{*}V)(\bx) = (DR(R^{-1}(\bx))V(R^{-1}(\bx))  .
\eeq
The pullback operator is
\[
    R^* = (R^{-1})_*  .
\]
The interior product, or contraction, of $\alpha$ with $V$ is the $(k-1)$-form
\[
    \imath_V \alpha = \alpha(V,\cdot,\ldots,\cdot)  .
\]
Suppose that $\vphi_{t_1,t_0}:M \times \bR^2 \to M$ is the ($C^1$) flow of a vector field $V(\bx,t)$, so that
$\vphi_{t_0,t_0}(\bx) = \bx$, and $\frac{d}{dt} \vphi_{t,t_0}(\bx) = V(\vphi_{t,t_0}(\bx), t)$. Then the Lie
derivative with respect to $V$ is the differential operator
\beq{LieDeriv}
    \cL_V (\cdot) \equiv  \left. \frac{d}{dt}\right|_{t=t_0}  \vphi_{t,t_0}^* (\cdot)  .
\eeq
The key identity for the derivative is Cartan's homotopy formula:
\beq{LieIdentity}
    \cL_V \alpha  \equiv \imath_V ( d \alpha ) + d(\imath_V \alpha)  .
\eeq
Note that $\cL$ behaves ``naturally" with respect to the pushfoward:
\beq{naturally}
	R_*\cL_V \alpha = \cL_{R_*V} R_*\alpha .
\eeq

\section{Remarks on Computing Lobe Volumes}\label{app:SpeedUp}

Here we comment on several technical aspects of our numerical implementation of the action-flux formulas.  We describe an efficient method for computing the intersection curves $\cI_\tau$, discuss examples in which the action-flux formulas may be applied even if the surface areas of lobe boundary components do not vanish as $t \to \pm \infty$, and remark on several ways to accelerate the convergence of the integrals in \Eq{pastAction} and \Eq{futureAction}.

Computing the curves $\cI_\tau =\partial\cP_\tau \cap \partial \cF_\tau$, as discussed in 
the examples of \Sec{abc} and \Sec{microdroplet}, essentially amounts to a root-finding and 
continuation problem.  The $\cI_\tau$ are defined to be zero-level-sets of some function on 
a 2D parameterization of $\partial \cP_0$.  This parameterization is sampled with a coarse 
grid, neighboring grid-points that bracket the zero-crossing are identified, and these 
brackets are refined along grid lines to produce ``seed'' points that lie on $\cI_\tau$.  A 
circle of radius $\delta$ (a pre-specified maximum Euclidean distance between neighboring 
curve points) is centered at each seed, and a 1D root finder on the angle around the circle 
is used to find a new point on $\cI_\tau$. The existence of this new point is guaranteed by 
the topology of $\cI_\tau$, provided the continued curve does not intersect the boundary of 
$M$.  A new $\delta$-circle is then centered at this new point, and the continuation is 
repeated until either the curve closes on itself, the angle between consecutive estimates 
of the curve tangent grows too large (this occurs when the curvature of $\cI_\tau$ is 
large), or the curve intersects the domain boundary.  In regions where the curvature of $
\cI_\tau$ is large, refinement is easily performed by reducing the radius $\delta$.  
Finally, care must be taken to ensure that a given seed  does not lie on a previously 
tracked curve.

Even if a well-defined lobe exists, and the curves $\cI_\tau$ are computed as described above, it may be the case that the $\alpha$-surface area of some component of the lobe boundary does not vanish in either direction of time under $\vphi^P$ and $\vphi^F$.  For example in \Sec{abc}, $\cU_0$ wraps entirely around $\Wu_0(p)$, and so its surface area never vanishes under $\vphi^P$; consequently, the action-flux formulas \Eq{pastAction}--\Eq{futureAction} can not be used to compute the second integral in \Eq{transformation}.  We can resolve this problem by dividing $\cU_0$ into subsurfaces whose $\alpha$-surface areas do eventually vanish under $\vphi^P$. For the case shown in \Fig{splitting}, there are four such subsurfaces, and it is easy to see that $\cU_0^1$ and $\cU_0^3$ collapse to $p$ as $t \to -\infty$, and  $\cU_0^2$ and $\cU_0^4$ collapse as $t \to +\infty$.  Using \Eq{pastAction} and \Eq{futureAction}, as appropriate, then gives
\beq{unstableIntegral}
  \int_{\cU_0} \alpha = \int_{-\infty}^{0}\left( \int_{\partial \cU_s^1 + \partial \cU_s^3} \lambda \right) ds 
                         - \int_{0}^{\infty} \left( \int_{\vphi^P_{s,0}(\partial \cU_0^2 + \partial \cU_0^4)} \lambda^P \right) ds .
\eeq
Here we emphasize that for the second integral the past flow $\vphi^P$ is to be used to evolve the boundaries, and the integrand is $\lambda^P$, the Lagrangian form for the past vector field $P$.

 \InsertFigTwo{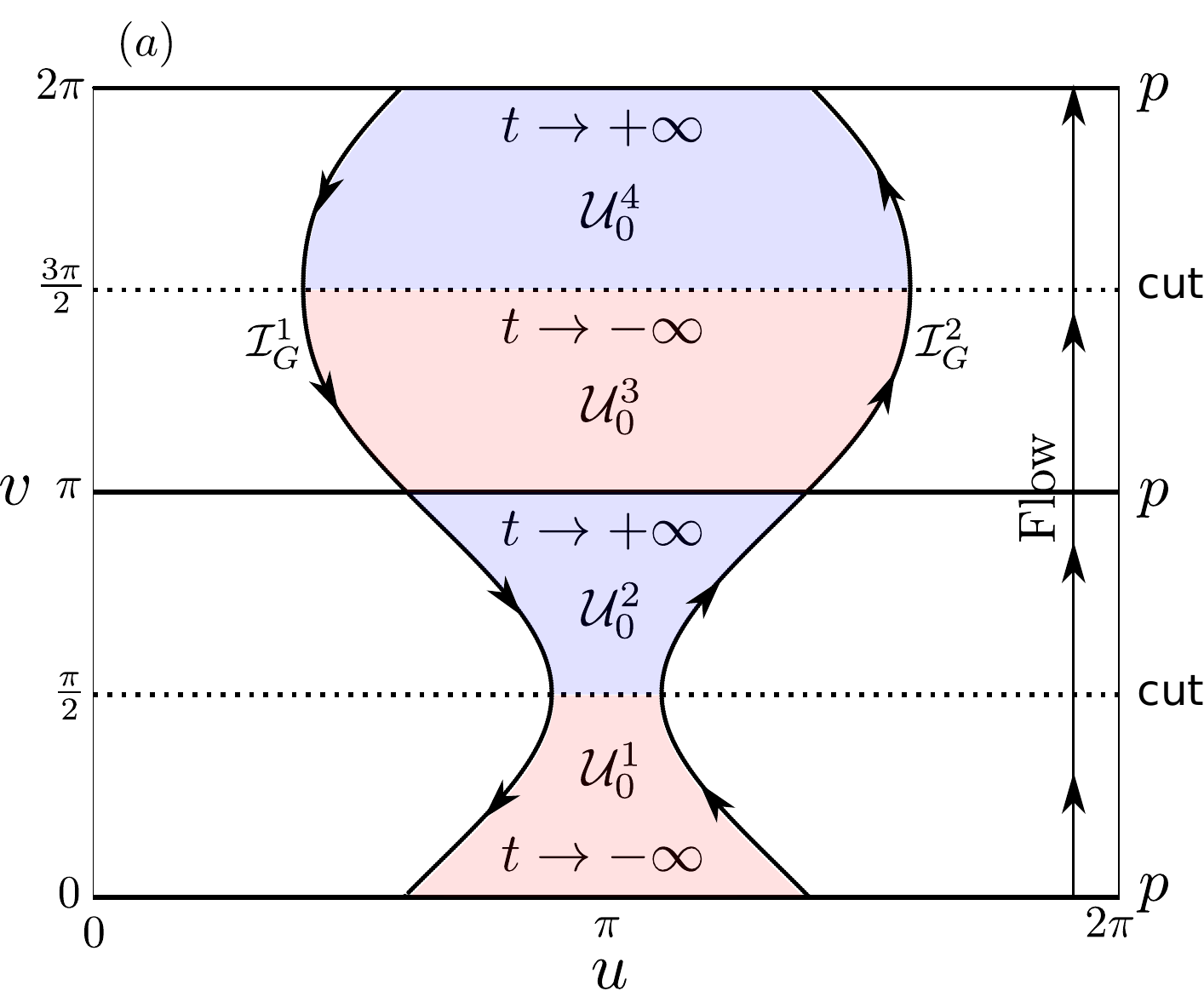}{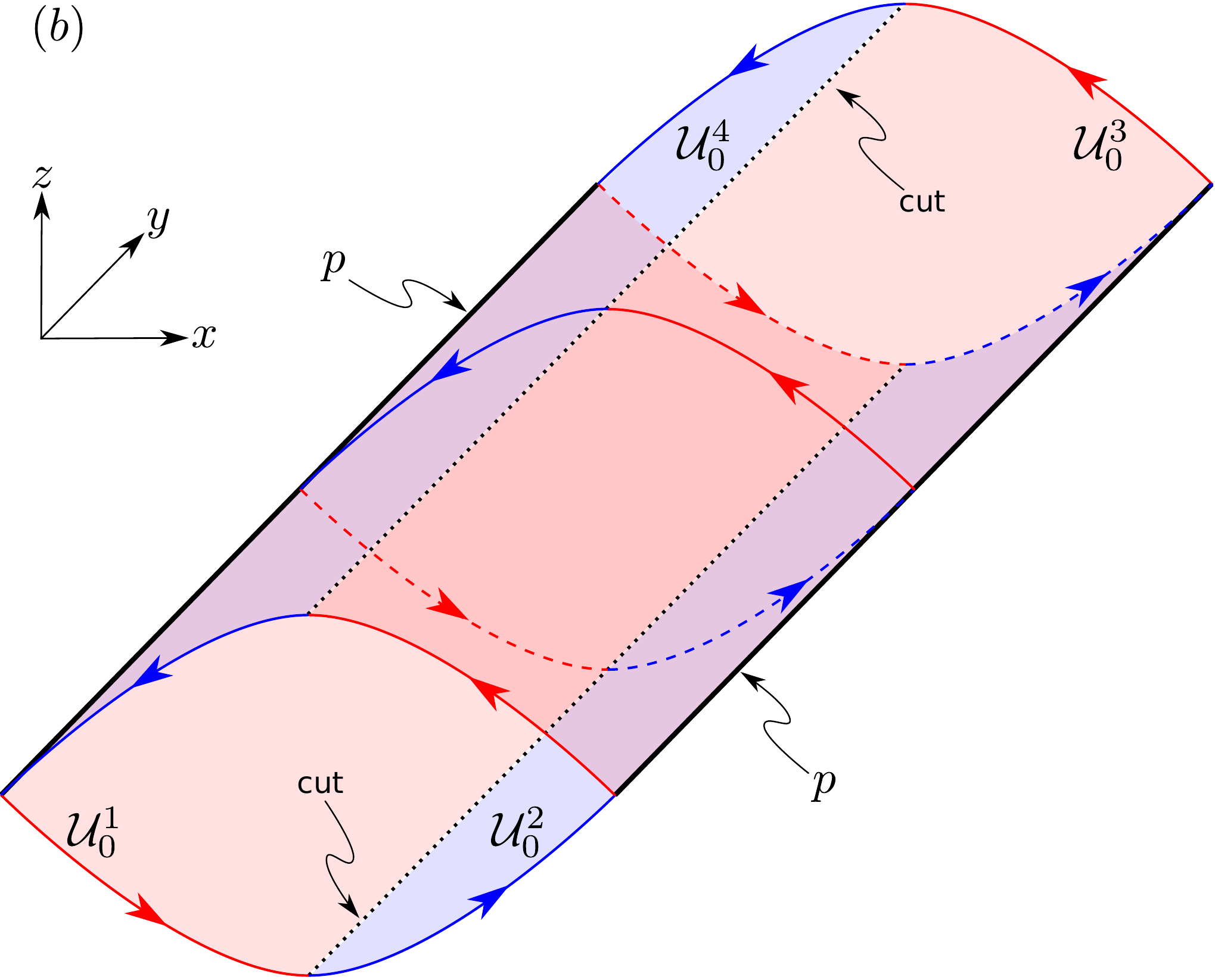}{(a) Subsurfaces of $\cU_0$, for the case shown in \Fig{Intersection2D_1}(a), that collapse under the past flow $\vphi^P$ as $t\to \infty$ (blue) or $t \to -\infty$ (red).  Arrows on $\cI_G$ denote orientation with respect to a right-handed outward normal.  (b) Subsurfaces $\cU^j_0 \subset M$.  Arrows denote direction of flow under $\vphi^P$.  Blue (red) regions collapse to $p$ in forward (backward) time.}{splitting}{.45\linewidth}


It is also possible to modify the application of \Eq{pastAction} and \Eq{futureAction} by choosing to evolve trajectories under a flow that does not coincide with the transitory flow $\vphi$.  Indeed, as we remarked in \Sec{lobeVolumes}, once a slice of codimension-one manifold is specified, \bTh{mainResult} is valid for its evolution under \emph{any} globally Liouville flow.  For example, the areas of $\cS_t^j$ in \Eq{stableIntegral} vanish under $\vphi^F$ both as $t \to \infty$ and as  $t \to -\infty$.  Thus, since $F$ itself is a globally Liouville vector field, \Eq{pastAction} may alternatively be applied with the replacements $\vphi \to \vphi^F$ and $\lambda \to \lambda^F$ to achieve the same answer as in \Eq{stableIntegral}; that is,
\beq{stableIntegralAlt}
  \int_{\cS_\tau^j} \alpha = \int_{-\infty}^\tau \left( \int_{\vphi^F_{s,\tau}(\cI_\tau^j)} \lambda^F \right) ds.
\eeq
We found it most efficient to choose either \Eq{stableIntegral} or \Eq{stableIntegralAlt} depending upon whether the intersection curve is closer to $f^1$ or $f^2$ at $t = \tau$.  In practice the intersection curves were closer to $f^1$, and so we used the past integral, \Eq{stableIntegralAlt}, for $\cS^1_\tau$  and the future, \Eq{stableIntegral}, for $\cS^2_\tau$. 

In addition to an appropriate choice of flow, the convergence of the time integrals in \Eq{stableIntegral} and \Eq{stableIntegralAlt} can be further accelerated by extrapolation.  Since the intersection curves lie on the stable manifolds of the periodic orbits $f^k$, they converge exponentially to these orbits in forward time; the contour integrals necessarily converge exponentially to zero in this same limit.  Thus, the convergence of the time integral can be accelerated by estimating the exponential tail of its integrand (a time-dependent contour integral) using the local expansion and contraction rates, $\sigma$ (for the ABC case $\sigma = \sqrt{AC}$), about the hyperbolic periodic orbits $f^k$ \Eq{fk}.
For example, the integral in \Eq{stableIntegral} can be truncated at time $t$ to give the estimate
\beq{expTail}
  \int_{\cS_\tau^j} \alpha \approx 
  - \int_\tau^t \left( \int_{\cI^j_s} \lambda \right) ds - \frac{1}{\sigma} \int_{\cI^j_t} \lambda.
\eeq
In practice we increase $t$ until this estimate converges to some desired tolerance.  Similar extrapolation is used to accelerate the integrals \Eq{BIntegral} in \Sec{microdroplet}.

Additional simplifications of the flux computation arise if the past or future vector fields are simple enough that explicit solutions can be found. For example, for the transitory ABC vector field, both $P$ and $F$ are integrable, and the analytical solutions on the separatrices are known \cite{Zhao93}.
Using these analytical formulas in the action-flux computations greatly reduces computation time. We use a similar simplification in \Sec{DropletComputations} for the flow on the droplet boundary.
When this simplification applies, the only numerical advection that that must be performed is over the transition interval $[0,\tau]$.

For the computations of \Sec{DropletComputations}, the separation of nearby trajectories affects the evaluation of \Eq{UIntegral} and \Eq{BIntegral} even when the intersection curves $\cI_G$ are numerically distinguishable. Indeed the distance between neighboring points on the numerically computed intersection curves grows nonuniformly. Consequently, for large $\tau$, the resolution of $\cI_\tau$ may be quite poor even if $\cI_0$ is well-resolved.  One way to ameliorate this effect is to use a second numerical representation of the intersection curves at time $t=\tau$ by solving a similar problem to \Eq{rootFindingProblem} for a parameterized representation of $\cS_\tau$. These curves can be integrated backward. Of course, in this case the resolution will degrade as $t$ decreases. If we use the two representations over the first and second halves of the transition interval, then the accuracy of the integrals \Eq{UIntegral} and \Eq{BIntegral} is improved.


Finally, if the transitory flow has symmetries, then these can be exploited to simplify the computations.  For example, the vector field \Eq{droplet} for the microdroplet example, with rotations \Eq{angles}, is reversible.  That is, there is a reversor $\Theta : M \to M$ such that
\beq{dropletReversor}
  \Theta_*V(\bx, t) = -V(\bx, \tau-t).
\eeq
This implies that $\Theta T = T^{-1} \Theta$, and
moreover, when $\cS_\tau = \cU_0$, that  $\cI_\tau = \Theta\cI_0$ \cite{MosovskyThesis}.  
Thus if we choose the numerical representations at $0$ and $\tau$ to respect this symmetry, the forward and backward iterations over the half-transition intervals are identical. This is much more efficient than solving the root-finding problem \Eq{rootFindingProblem} at both $t=0$ and $t=\tau$.

\bibliographystyle{siam}
\bibliography{3Dtransitory}

\end{document}